\documentclass[11pt]{amsart}
\usepackage{mathrsfs}
\usepackage{geometry}
\usepackage{graphicx}
\usepackage{amssymb}
\usepackage{epstopdf}
\usepackage{amsmath,amscd}
\usepackage{amsthm}
\usepackage{enumerate}
\usepackage{url,verbatim}
\usepackage{subfig}

\usepackage[normalem]{ulem}
\usepackage{fixme}

\usepackage{setspace}

\RequirePackage[colorlinks,citecolor=blue,urlcolor=blue]{hyperref}
\usepackage{breakurl}
\theoremstyle{plain}
\calclayout
\newtheorem{theorem}{Theorem}
\newtheorem{definition}[theorem]{Definition}
\newtheorem{lemma}[theorem]{Lemma}
\newtheorem{proposition}[theorem]{Proposition}

\newtheorem{example}[theorem]{Example}
\newtheorem{assumption}[theorem]{Assumption}

\newcommand\ol{\overline}

\newcommand\RR{{\mathbb R}}
\newcommand\ZZ{{\mathbb Z}}
\newcommand\NN{{\mathbb N}}

\newcommand\HH{{\mathbb H}}

\newcommand\si{\sigma}

\newcommand\resp{respectively}

\newcommand\Si{\Sigma}

\renewcommand\ell{l}

\newcommand\GT{\mathbb{G}\mathbb{T}}
\newcommand\CC{\mathbb{C}}
\newcommand\bm{\mathbf{m}}
\newcommand\SH{\mathrm{SH}}
\newcommand\YY{\mathbb{Y}}

\newcounter{mycount}

\numberwithin{equation}{section}
\numberwithin{theorem}{section}
\numberwithin{figure}{section}

\title[Limit shape of perfect matchings on contracting bipartite graphs]{Limit shape of perfect matchings on contracting bipartite graphs}

\date{}

\author{Zhongyang Li}
\address{Department of Mathematics,
University of Connecticut,
Storrs, Connecticut 06269-3009, USA}
\email{zhongyang.li@uconn.edu}
\urladdr{\url{https://mathzhongyangli.wordpress.com}}

\begin{document}
\maketitle

\begin{abstract}
We consider random perfect matchings on a general class of contracting bipartite graphs by letting certain edge weights be 0 on the contracting square-hexagon lattice in a periodic way. We obtain a deterministic limit shape in the scaling limit. The results can also be applied to prove the existence of multiple disconnected liquid regions for all the contracting square-hexagon lattices with certain edge weights, extending the results proved in \cite{ZL18} for contracting square-hexagon lattices where the number of square rows in each period is either 0 or 1.
\end{abstract}

\maketitle

\section{Introduction}

A dimer configuration, or a perfect matching of a graph is a choice of subset of edges such that each vertex is incident to exactly one edge. Dimer configurations on graphs form a natural mathematical model for the structure of matter, for example, the perfect matchings on the hexagonal lattice $\HH$ can describe the double-bond configurations in graphite molecules, where carbon atoms are represented by  vertices of $\HH$, and each double bond corresponds to a present edge in the perfect matching. One may assign each present edge in a perfect matching a non-negative weight depending on the energy of the double bond, and define the probability of a configuration to be proportional to the product of weights of present edges. See \cite{RK09} for an overview.

The limit behaviors of probability measures for perfect matchings on an infinite graph have been studied extensively. Unlike the well-known ferromagnetic Ising model, the limit measures of perfect matchings on infinite periodic bipartite graphs strongly depend on the boundary conditions. Indeed, for each slope of boundary conditions, one can construct a unique ergodic measure such that the measure is uniform conditional on the boundary slope. See \cite{KOS,ssrs}.

In this paper, we consider perfect matchings on a large class of bipartite graphs with a special type of boundary conditions, such that perfect matchings on such graphs form a Schur process, and the partition function (weighted sum of perfect matchings) can be computed by a Schur polynomial depending on edge weights and the bottom boundary condition. The dimer configurations on the class of bipartite graphs discussed here include uniform lozenge tilings of trapezoid domains (\cite{bg}), uniform domino tilings of rectangular domains (\cite{bk}), and perfect matchings on contracting square-hexagon lattices (\cite{BF15,BL17,ZL18,Li182,ZL20}) as special cases.  More precisely, uniform perfect matchings, or equivalently, random tilings on contracting square-hexagon lattices were first studied in \cite{BF15} to illustrate a shuffling algorithm.  The limit shapes and height fluctuations for perfect matchings with periodic edge weights (such that the probability distribution is not necessarily uniform) on contracting square-hexagon lattices with staircase boundary conditions were studied in \cite{BL17}. It is further proved in \cite{ZL18} that when certain edge weights converge to 0 exponentially fast with respect to the size of the graph and the bottom of the graph has piecewise boundary conditions, the liquid region in the limit shape splits into disconnected components when there is at most one row of squares in each period. In this paper, we study perfect matchings on a general class of bipartite graphs obtained by allowing certain edge weights of the contracting square-hexagon lattice to be 0. We shall show that the limit shape of random perfect matchings is deterministic and explicitly prove the equation for frozen boundaries (boundaries separating the liquid region and the frozen region). This way we obtain a 2D analogue of the Law of Large Numbers. The idea to study limit shape here is to analyze the asymptotics of the Schur polynomials at a general point using a formula obtained in \cite{ZL18}. Limit shape of perfect matchings can also be obtained by the variational principle; see \cite{CKP,KO07,ADP20}.

The existence of multiple disconnected components of liquid regions for the limit shape of perfect matchings with certain edge weights converging to 0 exponentially and piecewise bottom boundary conditions on a contracting square-hexagon lattice, in which there are at most one row of squares in each period, was proved in \cite{ZL18}. %For a contracting hexagon lattice (no square rows in each period) with certain edge weights, if in the scaling limit there are multiple disconnected liquid regions for perfect matchings, it was proved in \cite{ZL18} that each liquid region is the same as that of uniform lozenge tilings.
%For a contracting square-hexagon lattice with exactly one row of squares in each period, if in the scaling limit there are multiple disconnected liquid regions for perfect matchings, it was proved in \cite{ZL18} that one liquid region is the same as that of domino tilings, and all the others are the same as that of uniform lozenge tilings. 
The results proved in this paper can be applied to prove the existence of multiple disconnected liquid regions for limit shape of perfect matchings on an arbitrary contracting square-hexagon lattice with certain edge weights converging to 0 exponentially and piecewise boundary conditions. When multiple disconnected liquid regions in the limit shape occur, one component of the liquid state turns out to be exactly the liquid region of the limit shape of dimer configurations on a contracting bipartite graph studied in this paper, while all the other components are the same as liquid regions of the limit shape of dimer configurations on a contracting hexagon lattice (lozenge tilings).   

The organization of the paper is as follows. In Section \ref{sect:s2}, we define the contracting bipartite graph and the dimer model, and review related known results about the dimer partition function and the Schur polynomial. In Section \ref{sect:s3}, we prove an integral formula for the deterministic limit shape of dimer models on the contracting bipartite graphs, as well as the equations of the frozen boundary separating different phases in the limit shape.
In Section \ref{sect:s4}, we prove the existence of multiple disconnected liquid regions for the limit shape of perfect matchings on any contracting square-hexagon lattices with certain edge weights, extending the results in \cite{ZL18}.

\section{Background}\label{sect:s2}

 In this section, we define the contracting bipartite graph and the dimer model, and review related known results about the dimer partition function and the Schur polynomial.

\subsection{Contracting Bipartite Graphs}

For a positive integer $K$, let 
\begin{eqnarray*}
[K]:=\{1,2,\ldots,K\}.
\end{eqnarray*}

Consider a doubly-infinite binary sequence indexed by integers
$\ZZ=\{\ldots,-2,-1,0,1,2,\ldots\}$.
\begin{equation}
  \check{a}=(\ldots,a_{-2},a_{-1},a_0,a_1,a_2,\ldots)\in\{0,1\}^{\ZZ}.\label{ca}
\end{equation}

The \emph{whole-plane square-hexagon lattice} $\mathrm{SH}(\check{a})$ associated with the  sequence
$\check{a}$ is defined as follows. The vertex set of $\mathrm{SH}(\check{a})$ is a subset of $\frac{\ZZ}{2}\times \frac{\ZZ}{2} $.
Each vertex of $\mathrm{SH}(\check{a})$, represented by a point in the plane with coordinate in $\frac{\ZZ}{2}\times \frac{\ZZ}{2}$, is colored by either black or white.
 For $m\in \ZZ $, the black vertices have $y$-coordinate $m$; while the white
 vertices have $y$-coordinate $m-\frac{1}{2}$. We will label all the vertices
 with $y$-coordinate $t$ $\left(t\in\frac{\ZZ}{2}\right)$ as vertices in the $(2t)$th row. We further
 require that for each $m\in\ZZ$,
 \begin{itemize}
\item  each black vertex on the $(2m)$th row is adjacent to two white vertices in the $(2m+1)$th row; and
\item if $a_m=1$, each white vertex on the $(2m-1)$th row is adjacent to exactly one black vertex in the $(2m)$th row; 
 if $a_m=0$, each white vertex on the $(2m-1)$th row is adjacent to two black vertices in the $(2m)$th row. 

  \end{itemize}
 See Figure \ref{lcc}.

\begin{figure}
\subfloat[Structure of $\mathrm{SH}(\check{a})$ between the $(2m)$th row and the $(2m+1)$th row]{\includegraphics[width=.6\textwidth]{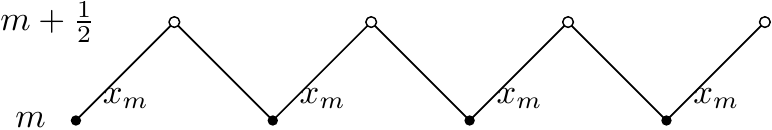}}\\
\subfloat[Structure of $\mathrm{SH}(\check{a})$ between the $(2m-1)$th row and the $(2m)$th row when $a_m=0$]{\includegraphics[width = .6\textwidth]{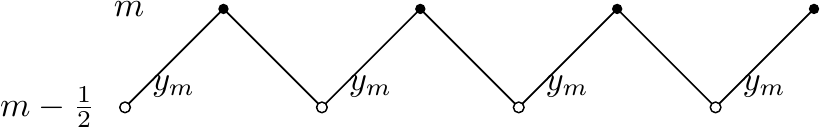}}\\
\subfloat[Structure of $\mathrm{SH}(\check{a})$ between the $(2m-1)$th row and the $(2m)$th row when $a_m=1$]{\includegraphics[width = .55\textwidth]{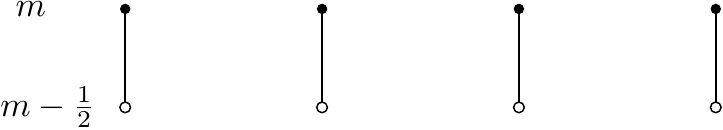}}
\caption{Graph structures of the square-hexagon lattice on the $(2m-1)$th, $(2m)$th, and $(2m+1)$th rows depend on the values of $(a_m)$. Black vertices are along the $(2m)$th row, while white vertices are along the $(2m-1)$th and $(2m+1)$th row.}
\label{lcc}
\end{figure}

Note that in the graph $\SH(\check{a})$, either all the faces on a row are hexagons, or all the faces on a row are squares, depending on the corresponding entry of $\check{a}$. 

We assign edge weights to $\SH(\check{a})$ as follows.

\begin{assumption}\label{apew}
\begin{enumerate}
\item  For $m\in\ZZ$, we assign weight $x_m\geq 0$ to each NE-SW edge joining the
  $(2m)$th row  to the $(2m+1)$th row of $\mathrm{SH}(\check{a})$. We assign
  weight $y_m>0$ to each NE-SW edge joining the $(2m-1)$th row to the $(2m)$th
  row of $\mathrm{SH}(\check{a})$, if such an edge exists; otherwise let $y_m=0$. We assign weight $1$
  to all the other edges. 
\item There exists a fixed positive integer $n$, such that for any $i,j\in \ZZ$ satisfying
\begin{eqnarray*}
[(i-j)\mod n]=0,
\end{eqnarray*} 
we have
\begin{eqnarray*}
x_i=x_j;\qquad y_i=y_j;\qquad a_i=a_j
\end{eqnarray*}
In other words, the graph is periodic with period $1\times n$; and the edge weights are assigned periodically with period $1\times n$.
\item There exists $\gamma\in [0,1), J\subset[n]$, such that 
\begin{enumerate}
\item $\gamma n\in\{0,1,2,\ldots,n-1\}$; and
\item $|J|=\gamma n$; and
\item $x_j=0$, for all $j\in J$;
\item $x_{i}=x>0$, for all $i\in [n]\setminus J$.
\end{enumerate}
 \end{enumerate} 
\end{assumption}

After removing all the edges with weight 0, we obtain a bipartite graph denoted by $\SH(\check{a},X,Y,n)$, where
\begin{eqnarray}
X&=&(\ldots,x_{-1},x_0,x_1,\ldots,);\label{dX}\\
Y&=&(\ldots,y_{-1},y_0,y_1,\ldots);\label{dY}
\end{eqnarray}
are edge weights.

A \emph{contracting bipartite graph} is built from a whole-plane
lattice as follows:

\begin{definition}\label{dfr}
  Let $N\in \NN$. Let $\Omega=(\Omega_1,\ldots,\Omega_N)$ be an $N$-tuple of
  positive integers, such that  $1=\Omega_1<\Omega_2<\cdots<\Omega_{N}$. Set
  The contracting square-hexagon lattice $\mathcal{R}(\Omega,\check{a})$ is a
  subgraph of $\mathrm{SH}(\check{a})$ with $2N$ or $2N+1$ rows of vertices.
  We shall now enumerate the rows of $\mathcal{R}(\Omega,\check{a})$
  inductively, starting from the bottom as follows:
  \begin{itemize}
    \item The first row consists of $N$ vertices $(i,j)$ with $i=\Omega_1-\frac{1}{2},\ldots,\Omega_N-\frac{1}{2}$ and $j=\frac{1}{2}$. We call this row the boundary row of $\mathcal{R}(\Omega,\check{a})$.
    \item When $k=2s$, for $s=1,\ldots N$,  the $k$th row consists of vertices $(i,j)$ with $j=\frac{k}{2}$ and incident to at least one vertex in the $(2s-1)th$ row of the whole-plane square-hexagon lattice $\mathrm{SH}(\check{a})$ lying between the leftmost vertex and rightmost vertex of the $(2s-1)$th row of $\mathcal{R}(\Omega,\check{a})$
    \item When $k=2s+1$, for $s=1,\ldots N$,  the $k$th row consists of vertices $(i,j)$ with $j=\frac{k}{2}$ and incident to two vertices in the $(2s)$th row of  of $\mathcal{R}(\Omega,\check{a})$.
  \end{itemize}
\end{definition}

 Let $\mathcal{R}(\Omega,\check{a},X,Y,n)$ be the corresponding weighted graph with edge weights satisfying Assumption \ref{apew}. Again we remove all the weight-0 edges in $\mathcal{R}(\Omega,\check{a},X,Y,n)$. See Figures \ref{fig:SH1} and \ref{fig:SH2} for examples.

  \begin{definition}
    \label{defI1I2}
  Let $I_1$ (resp.\@ $I_2$) be the set of indices $j$ such that vertices
  of the $(2j-1)$th row are connected to one vertex (resp.\@ two vertices) of
  the $(2j)$th row.
  In terms of the sequence $\check{a}$,

  \begin{equation*}
    I_1=\{k\in [N]\ |\ a_k=1\},\quad
    I_2=\{k\in [N]\ |\ a_k=0\}.
  \end{equation*}
\end{definition}

  The sets $I_1$ and $I_2$ form a partition of $[N]$, and we have
  $|I_1|=N-|I_2|$.
  
  \begin{example}\label{excb}Figure \ref{fig:SH} shows a contracting square-hexagon lattice with $(a_1,a_2,a_3)=(1,0,1)$ and $\Omega=(1,3,6)$. 
  
  When $x_2=0$, and $x_1=x_3=x>0$, we obtain a contracting bipartite graph as shown in Figure \ref{fig:SH1}.
  
    When $x_2=x_3=0$, and $x_1=x>0$, we obtain a contracting bipartite graph as shown in Figure \ref{fig:SH2}.
  
  \begin{figure}
  \includegraphics{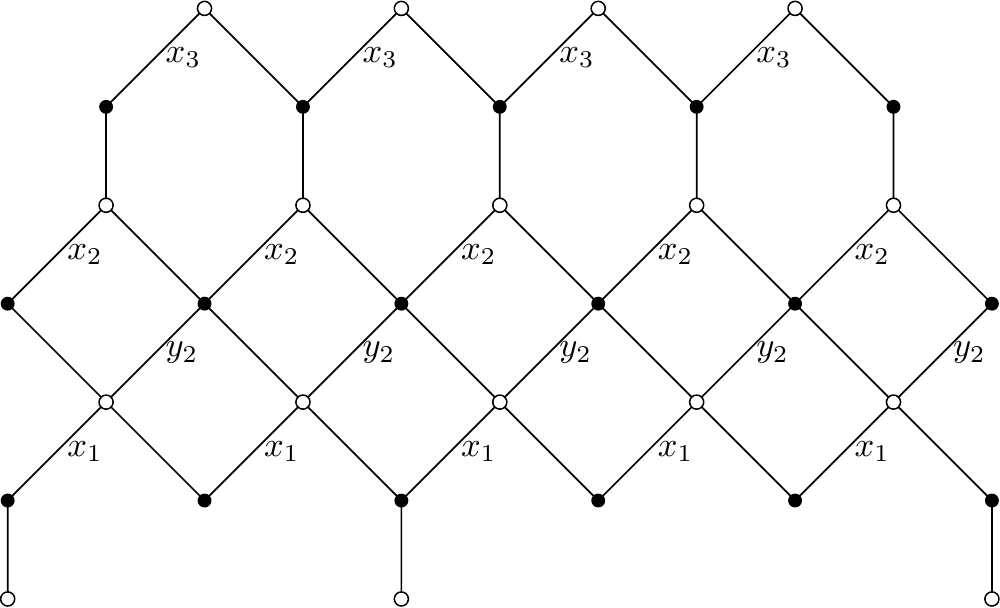}
\caption{Contracting square-hexagon lattice with $N=3$, $m=3$, $\Omega=(1,3,6), (a_1,a_2,a_3)=(1,0,1)$.}
\label{fig:SH}
\end{figure}

\begin{figure}
  \includegraphics{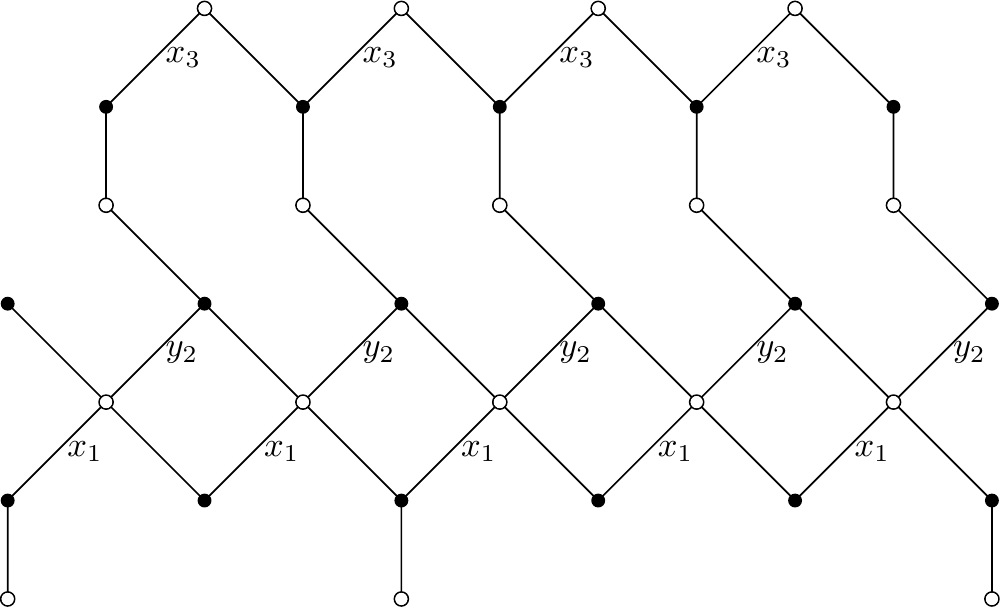}
\caption{Contracting bipartite lattice with $N=3$, $m=3$, $\Omega=(1,3,6), (a_1,a_2,a_3)=(1,0,1)$, $x_2=0$.}
\label{fig:SH1}
\end{figure}

\begin{figure}
  \includegraphics{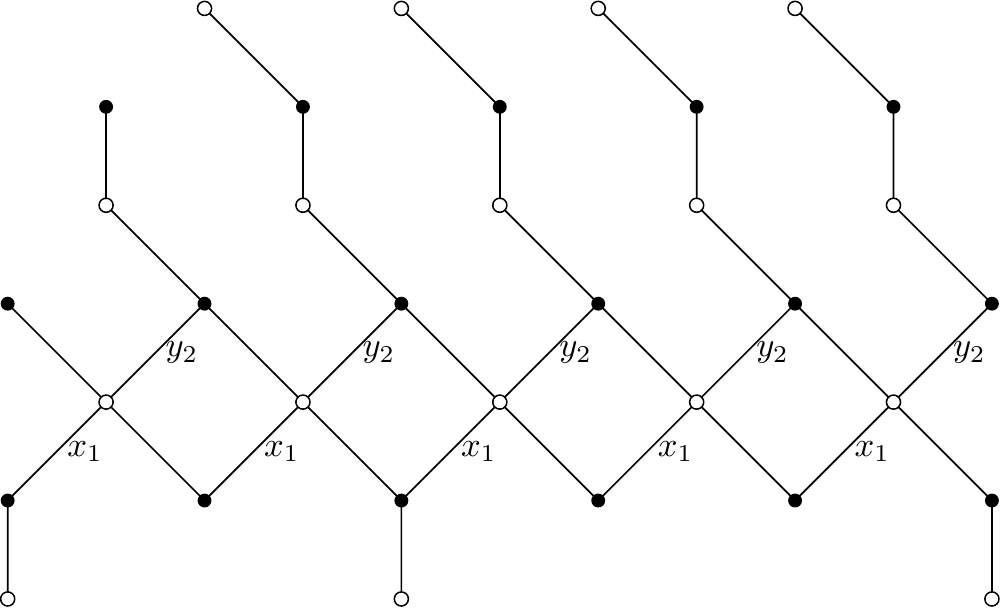}
\caption{Contracting bipartite lattice with $N=3$, $m=3$, $\Omega=(1,3,6), (a_1,a_2,a_3)=(1,0,1)$, $x_2=x_3=0$.}
\label{fig:SH2}
\end{figure}
\end{example}

\subsection{Partitions, counting measure and Schur functions}

\begin{definition}
  A \emph{partition} of length $N$ is a sequence of nonincreasing, nonnegative integers
  $\mu=(\mu_1\geq \mu_2\geq \ldots \geq\mu_N\geq 0)$. Each $\mu_k$ is a
  \emph{component} of the partition $\mu$. The length $N$ of the partition $\mu$
  is denoted by $\ell(\mu)$. The \emph{size} of a partition $\mu$ is
  \begin{equation*}
    |\mu| = \sum_{i=1}^N \mu_i.
  \end{equation*}

  We denote by $\YY_N$ the
  subset of length-N partitions.
\end{definition}

A graphic way to represent a partition $\mu$ is through its
\emph{Young diagram} $Y_\mu$, a collection of $|\mu|$ boxes arranged on
non-increasing rows aligned on the left: with
$\mu_1$ boxes on the first row, $\mu_2$ boxes on the second row,\dots $\mu_N$
boxes on the $N$th row. Some rows may be empty if the corresponding $\mu_k$ is
equal to 0. The correspondence between partitions of length $N$ and
Young diagrams with $N$ (possibly empty) rows is a bijection.

\begin{definition}
 We say that two partitions $\lambda$ and $\mu$ \emph{interlace}, or equivalently, their corresponding Young diagrams $Y_\lambda\subset Y_\mu$ differ by a horizontal
  strip, and
  write $\lambda \prec \mu$ if
\begin{eqnarray*}
\mu_i\geq \lambda_i\geq \mu_{i+1},\ \forall i\in\NN,
\end{eqnarray*}  
 where we assume $\lambda_i=0$ for all $i\geq l(\lambda)$ and $\mu_j=0$ or all $j\geq l(\mu)$.
   We say they \emph{co-interlace} and write $\lambda\prec'\mu$ if
  $\lambda'\prec \mu'$.
\end{definition}

\begin{definition}
  Let $\lambda\in \YY_N$. The \emph{rational Schur function} $s_\lambda$
  associated to $\lambda$ is the homogeneous symmetric function of degree
  $|\lambda|$ in $N$ variables defined as follows
  \begin{enumerate}
  \item If $N=1$, and $\lambda=(\lambda_1)$ then
  \begin{eqnarray*}
  s_{\lambda}(u_1)=u_1^{\lambda_1}.
  \end{eqnarray*}
  \item If $N\geq 2$, and $\lambda=(\lambda_1\geq \lambda_2\geq\ldots\geq \lambda_N\geq 0)$, then
\begin{equation*}
  s_{\lambda}(u_1,\ldots,u_N)=
  \frac{\det_{i,j=1,\ldots,N}(u_i^{\lambda_j+N-j})}{\prod_{1\leq
  i<j\leq N}(u_i-u_j)}.
\end{equation*}
\item Assume $(0,\ldots,0)$ consists of $N$ 0's. Then
\begin{eqnarray*}
s_{(0,\ldots,0)}(u_1,\ldots,u_N)=1
\end{eqnarray*}
\end{enumerate}
\end{definition}

It is straightforward to check that the Schur function defined above is a symmetric polynomial in the variables $(u_1,\ldots,u_N)$.

\begin{definition}Let $Y_{\lambda}$ be  a Young diagram of shape $\lambda$, drawn in the plane such that there are $\lambda_1$ squares on the top row, $\lambda_2$ squares on the 2nd top row, \ldots, and $\lambda_N$ squares on the $N$-th top row (bottom row). The squares in $Y_{\lambda}$ are indexed by $(i,j)$ with $i$ denoting the row number starting from the top and $j$ denoting the column number starting from the left. We may consider the Young diagram as a set consisting of all the squares indexed by such $(i,j)$'s, i.e.,
\begin{eqnarray*}
Y_{\lambda}:=\{(i,j)\in \NN^2: i\in [N], \lambda_i>0,  j\in [\lambda_i]\}.
\end{eqnarray*}

  A  semi-standard Young tableau (SSYT) of shape $\lambda$ is a map $T: Y_{\lambda}\rightarrow \NN$, which assigns a unique positive integer to each square in $Y_{\lambda}$, such that 
  \begin{eqnarray*}
  T(i,j)\leq T(i,j+1);\qquad\mathrm{and}\ T(i,j)<T(i+1,j).
  \end{eqnarray*}
\end{definition}

It is well-known that the Schur polynomial $s_{\lambda}(x_1,\ldots,x_N)$ can be combinatorially interpreted as a sum over SSYT of shape $\lambda$. 

\begin{proposition}\label{p28}
For a Young tableau $T$, let $sh(T)$ be the Young diagram (with 0 components removed) denoting the shape of $T$. Then
  \begin{eqnarray}
  s_{\lambda}(x_1,\ldots,x_N)=\sum_{T: sh(T)=\lambda}\prod_{(i,j)\in Y_{\lambda}}x_{T(i,j)}\label{shst}
  \end{eqnarray}
  where we assume $x_j=0$ for all $j>N$.
\end{proposition}

Let $\lambda\in\YY_N$ be a partition of length $N$. We define the
\emph{counting measure} $m(\lambda)$ corresponding to $\lambda$, which is a probability measure on $\left[0,1+\frac{\lambda_1}{N}\right)$, as follows.
\begin{equation}
m(\lambda)=\frac{1}{N}\sum_{i=1}^{N}\delta\left(\frac{\lambda_i+N-i}{N}\right).\label{ml}
\end{equation}

Let $\lambda(N)\in \YY_N$. Let $\Sigma_N$ be the permutation group of $N$ elements and let $\sigma\in \Sigma_N$. 
Assume that there exists a positive integer $n\in[N]$ such that $x_1,...,x_n$ are pairwise distinct and $\{x_1,...,x_n\}=\{x_1,...,x_N\}$. For $j\in[N]$, let
\begin{eqnarray}
\eta_j^{\sigma}(N)=|\{k\in[N]:k>j,x_{\sigma(k)}\neq x_{\sigma(j)}\}|.\label{et}
\end{eqnarray}
For $i\in[n]$, let
\begin{eqnarray}
\Phi^{(i,\sigma)}(N)=\{\lambda_j(N)+\eta_j^{\sigma}(N):x_{\sigma(j)}=x_i,j\in[N]\}\label{pis}
\end{eqnarray}
and let $\phi^{(i,\sigma)}(N)$ be the partition with length $|\{j\in[N]: x_j=x_i\}|$ obtained by decreasingly ordering all the elements in $\Phi^{(i,\sigma)}(N)$. 

\subsection{Dimer model}

\begin{definition}\label{dfvl}A dimer configuration, or a perfect matching $M$ of a finite graph is a set of edges  such that each vertex of $G$ belongs to an unique edge in $M$.

\end{definition}

\begin{definition} The partition function of the dimer model of a finite graph
$G$ with edge weights $(w_e)_{e\in E(G)}$ is given by
\begin{equation*}
Z=\sum_{M\in \mathcal{M}}\prod_{e\in M}w_e,
\end{equation*}
where $\mathcal{M}$ is the set of all perfect matchings of $G$. The
Boltzmann dimer probability measure on $M$ induced by the weights $w$ is
thus defined by declaring that probability of a perfect matching is equal to
\begin{equation*}
  \frac{1}{Z}\prod_{e\in M} w_e.
\end{equation*}
\end{definition}

 The set of perfect matchings of $\mathcal{R}(\Omega,\check{a},X,Y,n)$ is denoted by
  $\mathcal{M}(\Omega,\check{a},X,Y,n)$. Note that the contracting bipartite lattice $\mathcal{R}(\Omega,\check{a},X,Y,n)$ has degree-2 vertices. One way to study dimer model on a graph $G$ with degree-2 vertices is as follows. Let $v$ be an arbitrary degree-2 vertex, and let $u$ and $w$ be the two neighboring vertices of $v$. Remove the vertex $v$ and its two incident edges $(u,v)$ and $(v,w)$, then identify the two vertices $u$ and $w$ to obtain a new graph $\hat{G}$. It is straight forward to check that perfect matchings on $G$ and $\hat{G}$ are in 1-1 correspondence. If both $(u,v)$ and $(v,w)$ has weight 1, then the partition function for dimer configurations on $G$ and $\hat{G}$ are equal. For the contracting bipartite lattice $\mathcal{R}(\Omega,\check{a},X,Y,n)$, we may do the same thing of removing vertices, edges and identifying vertices as above, however, this will change the scaling limit of the graph and therefore obtain a different limit shape. In this paper, we shall always study perfect matchings on the graph $\mathcal{R}(\Omega,\check{a},X,Y,n)$ without the above manipulations and use the limit shape result to prove existence of multiple disconnected liquid regions on any contracting square-hexagon lattice with certain edge weights.

\begin{definition}\label{dlv}
Let $M$ be  a  perfect  matching  of the contracting bipartite graph $\mathcal{R}(\Omega,\check{a},X,Y,n)$ before removing all the edges with weight $x_i=0$.  We  call  a present  edge $e= ((i_1,j_1),(i_2,j_2))$ in $M$ a $V$-edge if $\max\{j_1,j_2\} \in\NN$ (i.e. if its higher extremity is black) and we call it a $\Lambda$-edge otherwise. In other words, the edges going upwards starting from an odd row are $V$-edges and those ones starting from an even row are $\Lambda$-edges. We also call the corresponding vertices-$(i_1,j_1)$ and $(i_2,j_2)$ of a $V$-edge (resp.\ $\Lambda$-edge) $V$-vertices (resp. $\Lambda$-vertices).
\end{definition}

There is a bijection between dimer configurations on the contracting bipartite graph $\mathcal{R}(\Omega,\check{a},X,Y,n)$, before removing the weight-0 edges, and sequences of
interlacing partitions. More precisely:
\begin{lemma}[\cite{BL17} Theorem 2.10,~\cite{bbccr}] Let $\mathcal{M}(\Omega,\check{a})$ the set of  all the  perfect  matching  of the contracting bipartite graph $\mathcal{R}(\Omega,\check{a},X,Y,n)$ before removing all the edges with weight $x_i=0$, and let $M\in \mathcal{M}(\Omega,\check{a})$. Then each vertex of $\mathcal{R}(\Omega,\check{a},X,Y,n)$ is either a $\Lambda$-vertex or a $V$-vertex with respect to $M$ by Definition \ref{dlv}.
  For given $\Omega=(\Omega_1,\ldots,\Omega_N)$, $\check{a}$, let $\omega\in \YY_N$ be the partition associated to
  $\Omega$ given by
\begin{eqnarray}
\omega=(\Omega_N-N,\Omega_{N-1}-N+1,\ldots,\Omega_1-1)\label{om}
\end{eqnarray}  
\begin{itemize}
\item For $0\leq i\leq (N-1)$ (resp.\ $1\leq j\leq N$), let $\mu^{(i)}\in \YY_i$ (\resp. $\nu^{(j)}$) be the partition associated to the $(2i+1)$th row (resp.\ $2j$th row) of vertices counting from the top. Assume
\begin{eqnarray*}
\mu^{(i)}&=&(\mu^{(i)}_1,\ldots,\mu^{(i)}_i);\\
\nu^{(j)}&=&(\nu^{(j)}_1,\ldots,\nu^{(j)}_j).
\end{eqnarray*}
Then for $1\leq k\leq i$, $\mu^{(i)}_k$ (resp.\ $1\leq k\leq j$, $\nu^{(j)}_k$) is the number of $\Lambda$-vertices on the left of the $k$th $V$-vertices, where the $V$-vertices are counting from the right.
\end{itemize}  
  Then we obtain a
  bijection between the set of perfect matchings
  $\mathcal{M}(\Omega,\check{a})$ and the set $S(\omega,\check{a})$ of
  sequences of partitions
  \begin{equation*}
    \{(\mu^{(N)},\nu^{(N)},\dots,\mu^{(1)}, \nu^{(1)}, \mu^{(0))}\}
  \end{equation*}
  where the signatures satisfy the following properties:
  \begin{itemize}
    \item All the parts of $\mu^{(0)}$ are equal to 0; and
    \item The partition $\mu^{(N)}$ is equal to $\omega$; and
    \item The partitions satisfy the following (co)interlacement relations:
      \begin{equation*}
        \mu^{(N)} \prec' \nu^{(N)} \succ \mu^{(N-1)} \prec' \cdots
        \mu^{(1)} \prec' \nu^{(1)} \succ \mu^{(0)}.
      \end{equation*}
  \end{itemize}
  Moreover, if $a_m=1$, then $\mu^{(N+1-k)}=\nu^{(N+1-k)}$.
  \label{myb}
\end{lemma}

\begin{proposition}(Proposition 2.15 of \cite{BL17})\label{p16}Let $\mathcal{R}(\Omega,\check{a},X,Y,n)$ be a contracting bipartite graph, whose edge weights satisfy Assumption \ref{apew}.
Then the partition function for perfect matchings on $\mathcal{R}(\Omega,\check{a},X,Y,n)$ is given by
\begin{eqnarray*}
Z=\left[\prod_{i\in I_2}\Gamma_i\right] s_{\omega}(x_{1},\ldots,x_{N})
\end{eqnarray*}
where $\omega\in \YY$ is the partition describing the bottom boundary condition  of $\mathcal{R}(\Omega,\check{a},X,Y,n)$ given by (\ref{om}).
Moreover, for $i\in I_2$, $\Gamma_i$ is defined by
\begin{eqnarray}
\Gamma_i=\prod_{t=i+1}^{N}\left(1+y_{i}x_{t}\right).\label{gi}
\end{eqnarray}.
\end{proposition}

\begin{example}\label{exl23}Figure \ref{fig:SH} shows a contracting square-hexagon lattice with $(a_1,a_2,a_3)=(1,0,1)$ and $\Omega=(1,3,6)$. Then the boundary partition is given by 
\begin{eqnarray*}
\omega=(\Omega_3-3,\Omega_2-2,\Omega_1-1)=(3,1,0)\in\YY_3
\end{eqnarray*}
By Proposition \ref{p16}, the partition function of the dimer configurations on the graph as given in Figure \ref{fig:SH} is
\begin{eqnarray*}
&&(1+y_2x_3)s_{(3,1,0)}(x_1,x_2,x_3)=(1+y_2x_3)\times\\
&&\left(x_1^3x_2 + x_1^3x_3 + x_1^2x_2^2 + 2x_1^2x_2x_3 + x_1^2x_3^2 + x_1x_2^3 + 2x_1x_2^2x_3 + 2x_1x_2x_3^2 + x_1x_3^3 + x_2^3x_3 + x_2^2x_3^2 + x_2x_3^3\right)
\end{eqnarray*}
\end{example}

\section{Limit Shape}\label{sect:s3}

In this section, we prove an integral formula for the deterministic limit shape of dimer models on the contracting bipartitie graphs, as well as the equations of the frozen boundary separating different phases in the limit shape. In Section \ref{sect:s31}, we prove formulas to compute Schur polynomials with some variables equal to 0, which is related to the partition function of dimer configurations on a contracting bipartitie graph, obtained from a contracting square-hexagon lattice by assigning some edge weights to be 0. In Section \ref{sect:s32}, we prove an explicit integral formula for the moments of the limit measure of dimer configurations on each horizontal level of the domain covered by the contracting bipartite graph; see Theorem \ref{p36}. In Section \ref{sect:s33}, we obtain the equation for the boundary curve separating different phases in the limit shape, and show that the boundary curve is an algebraic curve of a special type; more precisely (Theorem \ref{fb}), it is a cloud curve characterized by the numbers of intersections of its dual curve with arbitrary straight lines in $\mathbb{RP}^2$ (Proposition \ref{prop:cloud1}).

\subsection{Schur polynomials with vanishing variables.}\label{sect:s31}
\begin{lemma}\label{l210}Assume that $x_1,\ldots,x_N\in \CC$ such that
\begin{eqnarray}
|\{i\in [N]: x_i=0\}|=b\in [N]\label{bz}
\end{eqnarray}
Let $\lambda\in \YY_N$ such that there are exactly $a \in[N]$ components of $\lambda$ taking value 0. If $a<b$
 Then
\begin{eqnarray}
s_{\lambda}(x_1,\ldots,x_N)=0.\label{shz}
\end{eqnarray}
\end{lemma}

\begin{proof}The proof is based on the combinatorial interpretation of the Schur polynomial $s_{\lambda}(x_1,\ldots,x_N)$ as a sum over SSYT of shape $\lambda$ as specified in Proposition \ref{p28}. When $\lambda$ has exactly $a$ components taking value 0, $\lambda$ has exactly $N-a$ components which are strictly positive; and there are exactly $N-b$ numbers in $(x_1,x_2,\ldots,x_N)$ which are nonzero, denoted by $x_{j_1},x_{j_2},\ldots, x_{j_{N-b}}$ such that
 \begin{eqnarray*}
 j_1<j_2<\ldots<j_{N-b}.
\end{eqnarray*} 
 When $a<b$, we have $N-a>N-b$. In the first column of the Young diagram with shape $\lambda$, there are exactly $N-a$ squares. The integers $1,2,\ldots,N-b$ cannot fill the $N-a$ squares, hence in the strictly increasing sequence of integers in the first column of $Y_{\lambda}$, there must exist an integer $j$ such that $x_j=0$. This means that in the right hand side of (\ref{shst}), each summand is 0. Hence we obtain (\ref{shz}).
\end{proof}

%\begin{proposition}\label{p211}Let $k\in[N]$ and 
%\begin{eqnarray*}
%w_i:=\begin{cases}u_i&\mathrm{if}\ 1\leq i\leq k\\x_i&\mathrm{if}\ k+1\leq i\leq N. \end{cases}
%\end{eqnarray*}
%Assume that $(x_1,\ldots,x_N)$ takes $n$ distinct values $a_1,\ldots, a_n$. For $j\in[n]$, let
%\begin{eqnarray*}
%I_j=\{i\in[N]: x_i=a_j\};
%\end{eqnarray*}
%i.e., $I_j$ consisting of all the indices $i\in[N]$ such that $x_i=a_j$. Let
%\begin{eqnarray*}
%U_j=\{u_l\}_{l\in[k]\cap I_j}
%\end{eqnarray*}
%Then we have the following formula
%\begin{eqnarray}
%&&\label{sws}s_{\lambda}(w_1,\ldots,w_N)\\
%&=&\sum_{\ol{\sigma}\in[\Sigma_N/\Sigma_N^X]^r} \left(\prod_{j=1}^{n}s_{\phi^{(j,\sigma)}(N)}\left(U_j, a_j,\ldots,a_j\right)\right)\left(\prod_{i<j,x_{\sigma(i)}\neq x_{\sigma(j)}}\frac{1}{w_{\sigma(i)}-w_{\sigma(j)}}\right)\notag
%\end{eqnarray}
%where $\sigma\in \ol{\sigma}\cap \Sigma_N$ is a representative and
%\begin{eqnarray*}
%(U_j,a_j,\ldots,a_j)\in \CC^{|I_j|}.
%\end{eqnarray*}
%In particular, if $[k]\cap I_j=\emptyset$, then $U_j=\emptyset$, in this case by convention
%\begin{eqnarray*}
%(U_j,a_j,\ldots,a_j)=(a_j,\ldots,a_j)\in \CC^{|I_j|}
%\end{eqnarray*}
%\end{proposition}

%\begin{proof}When the $|I_j|=\frac{N}{n}$ for all $j\in [n]$, the formula (\ref{sws}) was proved in Proposition 3.4 of \cite{ZL18}. When $\{|I_j|\}_{j\in[n]}$ are not necessarily all equal, the proof follows from exactly the same arguments as the proof of Proposition 3.4 in \cite{ZL18}.
%\end{proof}

\begin{lemma}\label{l212}Let $N$ be a positive integer. Let $b\in[N]$. Assume that 
\begin{eqnarray*}
(u_1,\ldots,u_{N-b},0^{b})\in \CC^N
\end{eqnarray*}
where 
\begin{eqnarray*}
0^{b}=(0,\ldots,0)\in \RR^{b}
\end{eqnarray*}

Let $\lambda\in \YY_N$ such that there are exactly $a \in[N]$ components of $\lambda$ taking value 0.
If $a\geq b$, let 
\begin{eqnarray}
\tilde{\lambda}:&=&(\lambda_1,\lambda_2,\ldots,\lambda_{N-b})\in \YY_{N-b}\label{dtl}
\end{eqnarray}
 Then
\begin{eqnarray}
s_{\lambda}(u_1,\ldots,u_{N-b},0^b)=s_{\tilde{\lambda}}(u_1,\ldots,u_{N-b}).\label{sez}
\end{eqnarray}
\end{lemma}

\begin{proof}In (\ref{sez}), both $\lambda$ and $\tilde{\lambda}$
have the same Young diagram, and evaluating the last $b$ entries of the
left hand side to 0 is equivalent to forbidding integers $N-b+1,...,N$ to appear in
the SSYT. Then (\ref{sez}) follows.
\end{proof}

Assume $\psi\in \YY_m$, where $m$ is a positive integer. Recall that $s_{\psi}(1,\ldots,1)$ can be computed by the Weyl character formula as follows
\begin{eqnarray}
s_{\psi}(1,\ldots,1)=\prod_{1\leq i<j\leq m}\frac{\psi_i-\psi_j+j-i}{j-i}.\label{wf}
\end{eqnarray}

\begin{lemma}\label{lm216}Let $N$ be a positive integer. Assume that $x_1,\ldots,x_N\in \RR$ such that (\ref{bz}) and Assumption \ref{apew} (2)(3) hold with $b=\gamma N$. Let $\lambda\in \YY_N^+$ such that there are exactly $a \in[N]$ components of $\lambda$ taking value 0.
If $a\geq b$, let $\tilde{\lambda}\in \YY_{N-b}$ be defined as in (\ref{dtl}) and let
\begin{eqnarray}
\phi:&=&(\lambda_1+b,\lambda_2+b,\ldots,\lambda_{N-b}+b)\in \YY_{N-b}\label{dph}
\end{eqnarray}
Then 
\begin{eqnarray*}
s_{\lambda}(x_1,\ldots,x_N)=x^{|\lambda|}s_{\phi}(1,\ldots,1)=s_{\tilde{\lambda}}(x,\ldots,x)
\end{eqnarray*}
\end{lemma}

\begin{proof}The identity that $s_{\lambda}(x_1,\ldots,x_N)=s_{\tilde{\lambda}}(x,\ldots,x)$ follows from Lemma \ref{l212} by letting $u_1=u_2=\ldots=u_{N-b}=x$. Under the assumption that $a\geq b$, we obtain that
\begin{eqnarray*}
\lambda=(\tilde{\lambda},0,\ldots,0)\in \YY_N
\end{eqnarray*}
therefore $|\lambda|=|\tilde{\lambda}|$. Therefore,
\begin{eqnarray*}
s_{\tilde{\lambda}}(x,\ldots,x)=x^{|\tilde{\lambda}|}s_{\tilde{\lambda}}(1,\ldots,1)=x^{|\lambda|}s_{\tilde{\lambda}}(1,\ldots,1).
\end{eqnarray*}
By the Weyl formula (\ref{wf}), we obtain that 
\begin{eqnarray*}
s_{\phi}(1,\ldots,1)=s_{\tilde{\lambda}}(1,\ldots,1).
\end{eqnarray*}
Then the lemma follows.
\end{proof}

\begin{example}Consider Example \ref{exl23}.
When $x_2=0$, and $x_1=x_3=x>0$, let $b=1$, and then
\begin{eqnarray*}
\phi=(\omega_1+b,\omega_2+b)=(4,2).
\end{eqnarray*}
 we obtain the partition function of dimer configurations on the graph as given in Figure \ref{fig:SH1} is 
\begin{eqnarray*}
&&(1+y_2x_3)(x_1^3x_3 + x_1^2x_3^2 + x_1x_3^3)=3(1+y_2x)x^4=(1+y_2x)x^{|\omega|}s_{\phi}(1,1),
\end{eqnarray*}
where the last identity follows from (\ref{wf}).

When $x_2=x_3=0$ and $x_1=x>0$, then the partition function of dimer configurations on the graph as given in Figure \ref{fig:SH2} is 0. Indeed, Figure \ref{fig:SH2} does not admit a perfect matching.
\end{example}

\subsection{Limit counting measure}\label{sect:s32}

For computational simplicity, when studying the limit shape of perfect matchings, we shall consider perfect matchings on a graph by translation every row of $\mathcal{R}(\Omega,\check{a},X,Y,n)$ to start from the same vertical line. See Figure \ref{fig:shtl} for an example.

  \begin{figure}[htbp] 
\includegraphics*[width=0.3\hsize]{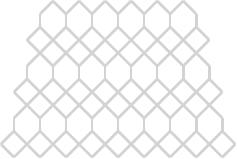}\qquad\qquad\includegraphics*[width=0.3\hsize]{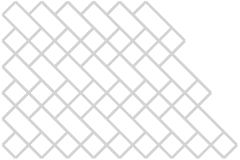}
\caption{A contracting square hexagon lattice: the left graph represents a subgraph of the original lattice; the right graph represents the lattice obtained by letting all the rows of the left graph start from the same vertical line.}\label{fig:shtl}
 \end{figure}

\begin{definition}[\cite{GP15}]
  \label{df23}
  A sequence of signatures $\lambda(N)\in\YY_N$ is called \emph{regular}, if there
  exists a piecewise continuous function $f(t)$ and a constant $C>0$ such that
  \begin{equation*}
    \lim_{N\rightarrow\infty} \frac{1}{N}
    \sum_{j=1}^{N}
    \left|\frac{\lambda_j(N)}{N}-f\left(\frac{j}{N}\right)\right| = 0,
    \quad
    \text{and}
    \quad
    \sup_{1\leq j\leq N}
    \left|\frac{\lambda_j(N)}{N}-f\left(\frac{j}{N}\right)\right|<C\ \text{for
    all $N\geq1$}.
  \end{equation*}
\end{definition}

If $\{\lambda(N)\in \YY_N\}_{N\in\NN}$ is a regular sequence of signatures, then the sequence of
counting measures $m(\lambda(N))$ converges weakly to a measure $\bm$ with
compact support.
By
Theorem~3.6 of~\cite{bk} there exists an explicit function
$H_{\bm}$, analytic in a neighborhood of 1,  depending on the weak
limit $\bm$ such that
\begin{equation}
  \lim_{N\rightarrow\infty}
  \frac{1}{N} \log\left(%
  \frac{s_{\lambda(N)}(u_1,\ldots,u_k,1,\ldots,1)}{s_{\lambda(N)}(1,\ldots,1)}
  \right) = H_{\bm}(u_1)+\cdots+H_{\bm}(u_k),
  \label{nlc}
\end{equation}
and the convergence is uniform when $(u_1,\dotsc,u_k)$ is in a neighborhood
of $(1,\dots,1)$. Here in $s_{\lambda(N)}(u_1,\ldots,u_k,1,\ldots,1)$, there are exactly $N-k$
1's, and in $s_{\lambda(N)}(1,\ldots,1)$, there are exactly $N$ 1's.

Precisely, $H_{\bm}$ is constructed as follows: let
$S_{\bm}(z)=z+\sum_{k=1}^\infty M_k(\bm) z^{k+1}$ be the moment generating
function of the measure $\bm$, where $M_k(\bm)=\int x^k d\bm(x)$, and
$S_{\bm}^{(-1)}$ be the inverse function of $S_{\bm}$. Let $R_{\bm}(z)$ be the
\emph{Voiculescu R-transform} of $\bm$ defined as
\begin{equation*}
  R_{\bm}(z) = \frac{1}{S_\bm^{(-1)}(z)} - \frac{1}{z}.
\end{equation*}
Then
\begin{equation}
  \label{hmz}
  H_{\bm}(u) = \int_{0}^{\ln u} R_\bm(t)dt+ \ln\left( \frac{\ln u}{u-1} \right).
\end{equation}
In particular, $H_{\bm}(1)=0$, and
\begin{equation*}
  H'_\bm(u) = \frac{1}{u S_\bm^{(-1)}(\ln u)} - \frac{1}{u-1}.
\end{equation*}

\begin{definition}
  \label{df21}
  Let
 \begin{eqnarray*}
 X_N=(x_1,x_2,\ldots,x_N)
 \end{eqnarray*} 
  
   Let $\rho$ be a
  probability measure on $\YY_N$. The \emph{Schur generating function}
  $\mathcal{S}_{\rho,X_N}(u_1,\ldots,u_{N})$ with respect to parameters
  $X_N$ is the symmetric Laurent series in $(u_1,\ldots,u_{N})$ given by
  \begin{equation*}
    \mathcal{S}_{\rho,X_N}(u_1,\ldots,u_{N})=
    \sum_{\lambda\in\YY_N} \rho(\lambda)
    \frac{s_{\lambda}(u_1,\ldots,u_{N})}{s_{\lambda}(X_N)},
  \end{equation*}
\end{definition}

%Let 
%\begin{equation*}
%V(u_1,\ldots, u_{(1-\gamma)N})=\prod_{1\leq %i<j\leq N}(u_i-u_j)
%\end{equation*}
%be the Vandermonde determinant with respect to variables $u_1,\ldots,u_{N}$.
%Introduce the family $(\mathcal{D}_k)$ of differential operators acting on
%symmetric functions $f$ with variables $u_1,\ldots, u_{N}$ as follows:
%\begin{equation}
%\mathcal{D}_k %f=\frac{1}{V}\left(\sum_{i=1}^{N}\left(u_i\frac{\p%artial}{\partial
%u_i}\right)^k\right)(V\cdot f)\label{dk}.
%\end{equation}

%For every $\phi\in\GT_{N}^+$,
%the Schur function $s_\phi(u_1,\dots,u_{N})$ is
%an eigenfunction of $\mathcal{D}_k$, associated with the
%eigenvalue $\sum_{i\in[N]} (\phi_i+N-i)^k$, see~\cite[Proposition~4.3]{bg}.

\begin{assumption}\label{ap33}Let $\mathcal{R}(\Omega,\check{a},X,Y,n)$ be a contracting bipartite graph with edge weights satisfying Assumption \ref{apew}. Assume the bottom boundary partition $\omega^{(N)}=(\omega^{(N)}_1,\omega^{(N)}_2,\ldots,\omega^{(N)}_N)\in \YY_N$ satisfies 
\begin{itemize}
\item there exists $\alpha\in[\gamma, 1)$ and $\alpha N\in[N]$, such that 
\begin{eqnarray*}
\omega^{(N)}_{(1-\alpha)N+1}=\omega^{(N)}_{(1-\alpha)N+2}=\ldots=\omega^{(N)}_{N}=0.
\end{eqnarray*}
\item Let 
\begin{eqnarray*}
\tilde{\omega}^{(N)}=(\omega_1^{(N)},\omega_2^{(N)},\ldots,\omega_{(1-\gamma)N}^{(N)}),
\end{eqnarray*}
then $\{\tilde{\omega}^{(N)}\}_{N\in\NN}$ form a regular sequence of partitions with counting measures converging to $\tilde{\bm}$ as $N\rightarrow\infty$.
\end{itemize}
\end{assumption}

For any integer $i\in \ZZ$, let
\begin{eqnarray*}
\bar{i}:=\begin{cases}n,&\mathrm{if\ }(i\mod n)=0;\\ (i\mod n),&\mathrm{otherwise}. \end{cases}
\end{eqnarray*}
Hence $\bar{i}\in[n]$. If the edge weights are periodic as in Assumption 2.1 (2), we have
\begin{eqnarray*}
x_i=x_{\bar{i}};\qquad y_i=y_{\bar{i}}
\end{eqnarray*}

\begin{lemma}(Lemma 3.6 of \cite{BL17})
  \label{lm212}
Assume the edge weights of $\mathcal{R}(\Omega,\check{a},X,Y,n)$ satisfy Assumption \ref{apew} (1)(2).  For any $k$ between 0 and $2N-1$, define $t=\lceil k/2\rceil$, and
  let 
  \begin{equation*}
    X^{(N-t)}=(x_{\ol{t+1}},\ldots,x_{\bar{N}}),
    \quad
    \text{and}
    \quad
    W^{(t)}=(x_{\bar{1}},\ldots,x_{\bar{t}}).
  \end{equation*}
 and $\rho^k$ be the probability measure of the partitions corresponding to dimer configurations on the $(2N-k)$th row, counting from the top.
  Then the Schur generating function $\mathcal{S}_{\rho^k,X^{(N-t)}}$ is given by:
\begin{equation*}
\mathcal{S}_{\rho^k,X^{(N-t)}}(u_1,\ldots,u_{N-t})=
\frac{s_{\omega}\left(u_1,\ldots,u_{N-t},W^{(t)}\right)}{s_{\omega}(X^{(N)})}
\prod_{i\in\{1,\ldots,t\}\cap I_2}
\prod_{j=1}^{N-t}\left(\frac{1+y_{\bar{i}}u_j}{1+y_{\bar{i}}x_{\ol{t+j}}}\right).
\end{equation*}
where $\omega$ is given by (\ref{om}).
\end{lemma}

When the edge weights are assigned periodically as in Assumption \ref{apew} (1)(2)(3), and the boundary partition $\omega$ satisfies Assumption \ref{ap33}, let 
\begin{eqnarray*}
\tilde{\omega}=(\omega_1,\omega_2,\ldots,\omega_{(1-\gamma)N})
\end{eqnarray*}
For each integer $k$ between $0$ and $2N-1$, we consider the partition corresponding to dimer configurations on the $(2N-k)$th row (the top row is the 1st row) of the contracting bipartite graph $\mathcal{R}(\Omega,\check{a},X,Y,n)$. When the edge weights satisfy Assumption 2.1 (3), i.e., when certain $x_i$'s are 0,  by Lemma \ref{l210}, any possible partition on the $(2N-k)$th row with strictly positive probability to occur falls into one of the following two cases.
\begin{enumerate}
\item If $k=2t-1$ for some $t\in[N]$, then the partition corresponding to the dimer configuration on the $(2N-2t+1)$th row satisfies $\mu^{(N-t)}\in \YY_{N-t}$ and
\begin{eqnarray*}
\mu^{(N-t)}_{\lfloor(1-\gamma)(N-t)\rfloor+1}=\mu^{(N-t)}_{\lfloor(1-\gamma)(N-t)\rfloor+2}=\ldots=\mu^{(N-t)}_{N-t}=0.
\end{eqnarray*}
Let
\begin{eqnarray*}
\tilde{\mu}^{(N-t)}:=\left(\mu_{1}^{(N-t)},\mu_2^{(N-t)},\ldots,\mu^{(N-t)}_{\lfloor(1-\gamma)(N-t)\rfloor}\right)\in \YY_{\lfloor (1-\gamma)(N-t) \rfloor}.
\end{eqnarray*}
\item If $k=2t$ for some $t\in\{0,1,\ldots,N-1\}$, then the partition corresponding to the dimer configuration on the $(2N-2t)$th row satisfies $\nu^{(N-t)}\in \YY_{N-t}$ and
\begin{eqnarray*}
\nu^{(N-t)}_{\lfloor(1-\gamma)(N-t)\rfloor+1}=\nu^{(N-t)}_{\lfloor(1-\gamma)(N-t)\rfloor+2}=\ldots=\nu^{(N-t)}_{N-t}=0.
\end{eqnarray*}
Let
\begin{eqnarray*}
\tilde{\nu}^{(N-t)}:=\left(\nu_{1}^{(N-t)},\nu_2^{(N-t)},\ldots,\nu^{(N-t)}_{\lfloor(1-\gamma)(N-t)\rfloor}\right)\in \YY_{\lfloor (1-\gamma)(N-t) \rfloor}.
\end{eqnarray*}
\end{enumerate}
Let $\tilde{\rho}^k$ (resp. $\rho^k$) be the probability measure on $\tilde{\mu}^{(N-t)}$ or $\tilde{\nu}^{(N-t)}$ (resp.\ $\mu^{(N-t)}$ or $\nu^{(N-t)}$).

Assume
\begin{eqnarray*}
u_{(1-\gamma)(N-t)+1}=u_{(1-\gamma)(N-t)+2}=\ldots=u_{N-t}=0.
\end{eqnarray*}
By Lemmas \ref{lm216} and \ref{lm212} we obtain
\begin{eqnarray*}
&&\mathcal{S}_{\rho^k,X^{(N-t)}}(u_1,\ldots,u_{(1-\gamma)(N-t)},0,\ldots,0)\\
&=&
\frac{s_{\tilde{\omega}}\left(u_1,\ldots,u_{(1-\gamma)(N-t)},x,\ldots,x\right)}{s_{\tilde{\omega}}(x,\ldots,x)}
\prod_{i\in\{1,\ldots,t\}\cap I_2}
\prod_{j=1}^{(1-\gamma)(N-t)}\left(\frac{1+y_{\bar{i}}u_j}{1+y_{\bar{i}}x}\right).
\end{eqnarray*}
Let $\tilde{X}^{(N-t)}$ consist of all the components of $X^{(N-t)}$ which are nonzero. Let
\begin{eqnarray*}
1^{\lfloor(1-\gamma)(N-t)\rfloor}=(1,\ldots,1)\in \RR^{\lfloor(1-\gamma)(N-t)\rfloor}
\end{eqnarray*}

Then
\begin{eqnarray}
&&\mathcal{S}_{\tilde{\rho}^k, 1^{\lfloor(1-\gamma)(N-t)\rfloor}}\left(u_1,\ldots,u_{\lfloor(1-\gamma)(N-t)\rfloor}\right)\label{tse}\\
&=&\sum_{\tilde{\lambda}\in \YY_{\lfloor(N-t)(1-\gamma)\rfloor}}\tilde{\rho}^k(\tilde{\lambda})
\frac{s_{\tilde{\lambda}}\left(u_1,\ldots,u_{\lfloor(1-\gamma)(N-t)\rfloor}\right)}{s_{\tilde{\lambda}}(1^{\lfloor(1-\gamma)(N-t) \rfloor})}\notag\\
&=&\sum_{\tilde{\lambda}\in \YY_{\lfloor(N-t)(1-\gamma)\rfloor}}\tilde{\rho}^k(\tilde{\lambda})
\frac{s_{\tilde{\lambda}}\left(u_1x,\ldots,u_{\lfloor(1-\gamma)(N-t)\rfloor}x\right)}{s_{\tilde{\lambda}}(\tilde{X}^{(N-t)})}\notag\\
&=&\sum_{\lambda\in \YY_{N-t}} \rho^k(\lambda)
\frac{s_{\lambda}\left(u_1x,\ldots,u_{\lfloor(1-\gamma)(N-t)\rfloor}x,0,\ldots,0\right)}{s_{\lambda}({X}^{(N-t)})}\label{le3}\\
&=&\mathcal{S}_{\rho^k,X^{(N-t)}}(u_1x,\ldots,u_{(1-\gamma)(N-t)}x,0,\ldots,0)\notag
\end{eqnarray}
In particular we have $\rho^k(\lambda)=0$ if $\lambda\in \YY_{N-t}$ is not an extension of $\tilde{\lambda}\in \YY_{\lfloor(1-\gamma)(N-t)\rfloor}$ by 0's, according to Lemma \ref{l210}. The expression (\ref{le3}) also follows from Lemma \ref{l212}.

Letting $N\to\infty$, $\frac{t}{N}\to \kappa\in[0,1)$, by (\ref{nlc}) we have
\begin{multline*}
\lim_{(1-\kappa)N\rightarrow\infty}
\frac{1}{(1-\kappa)N} \log\mathcal{S}_{\rho^k,X^{(N-t)}}(u_1x,\ldots,u_{\ell}x,x,\ldots,x,0,\ldots,0)\\
=\frac{1-\gamma}{1-\kappa}\sum_{1\leq i\leq \ell}\left[Q_{\kappa}(u_i)-Q_{\kappa}(1)\right],
\end{multline*}
%where $\kappa=\frac{t}{N}$, and $k=2t+1$ or $k=2t+2$. Let
where
\begin{equation*}
  Q_{\kappa}(u)=H_{\tilde{\bm}}\left(u\right)
    +\frac{\kappa}{n(1-\gamma)}
    \sum_{i\in [n]\cap I_2}\log(1+y_ix u).
\end{equation*}

By (\ref{tse}), we obtain
\begin{eqnarray*}
\lim_{N\rightarrow\infty}
\frac{1}{(1-\gamma)(1-\kappa)N} \log\mathcal{S}_{\tilde{\rho}^k, 1^{\lfloor(1-\gamma)(N-t)\rfloor}}\left(u_1,\ldots,u_l,1,\ldots,1\right)\\
=\frac{1}{1-\kappa}\sum_{1\leq i\leq \ell}\left[Q_{\kappa}(u_i)-Q_{\kappa}(x)\right],
\end{eqnarray*}

\begin{lemma}[\cite{bg}, Theorem~5.1]
  \label{l28}
  Let $(\rho_N)_{N\geq 1}$ be a sequence of measures such that for each
  $N$, $\rho_N$ is a probability measure on $\GT_N^+$, and 
 for every $j$, the following convergence holds uniformly in a complex
 neighborhood of $(1,\ldots,1)\in\mathbb{C}^j$
  \begin{equation}
    \lim_{N\rightarrow\infty} \frac{1}{N}
    \log\mathcal{S}_{\rho_N,1^N} \left(u_1,\ldots,u_j,1,\ldots,1\right)
    = Q(u_1)+\cdots +Q(u_j),
    \label{lm}
\end{equation}
 with $Q$ an analytic function in a neighborhood of $1$.
 Then the sequence of random measures $(m(\rho_N))_{N\geq 1}$ converges as
 $N\rightarrow\infty$ in
probability in the sense of moments to a deterministic measure $\mathbf{m}$ on
$\RR$, whose moments are given by
\begin{eqnarray*}
\int_{\RR}x^j\mathbf{m}(dx)&=&\sum_{l=0}^{j}\frac{j!}{l!(l+1)!(j-l)!}\left.\frac{\partial^l}{\partial
u^l}\left(u^jQ'(u)^{j-l}\right)\right|_{u=1}.\\
\end{eqnarray*}
\end{lemma}

By Lemma \ref{l28}, we obtain the following theorem about limit shape of perfect matchings on the contracting bipartite graph $\mathcal{R}(\Omega,\check{a},X,Y,n)$.

\begin{theorem}\label{p36}Let $k\in\{0,1,\ldots,2N-1\}$, such that $\lim_{N\rightarrow\infty}\frac{k}{2N}=\kappa$. Assume the edge weights and boundary partition of the contraction bipartite graph $\mathcal{R}(\Omega,\check{a},X,Y,n)$ satisfy Assumptions \ref{apew} and \ref{ap33}. Then the sequence of random measures $(m(\tilde{\rho}^k))_{N\in\NN}$ converges as
 $N\rightarrow\infty$ in
probability in the sense of moments to a deterministic measure $\tilde{\mathbf{m}}^{\kappa}$ on
$\RR$, whose moments are given by
\begin{eqnarray*}
\int_{\RR}x^j \tilde{\mathbf{m}}^{\kappa}(dx)=\frac{1}{2(j+1)\pi \mathbf{i}}\oint_1\frac{dz}{z}\left[F_{\kappa}(z)\right]^{j+1}
\end{eqnarray*}
where
\begin{eqnarray*}
F_{\kappa}(z)=\frac{z}{1-\kappa}H'_{\tilde{\bm}}(z)+\frac{\kappa z}{n(1-\kappa)(1-\gamma)}\sum_{i\in [n]\cap I_2}\frac{y_i x}{1+y_ixz}+\frac{z}{z-1}
\end{eqnarray*}
and the integration goes over a small positively oriented contour around 1.
\end{theorem}

\subsection{Frozen boundary}\label{sect:s33}
 The frozen region is defined to be the region where the density of the limit of the counting measures for $\rho^k$ is 0 or 1, while the liquid region is the region where this density is strictly between 0 or 1.  The frontier between the the frozen region and the liquid region is called the frozen boundary.

We consider a special case of bottom boundary conditions.

\begin{assumption}\label{ap37}
\begin{enumerate}
\item Let\begin{multline}
\label{cc1}
\Omega=
(A_1,A_1+1,\ldots,B_1-1,B_1,A_2,A_2+1,\ldots,%B_2-1,
B_2,\ldots,A_s,A_s+1,\ldots,%B_s-1,
B_s),
\end{multline}
where
\begin{eqnarray*}
\sum_{i=1}^{s}(B_i-A_i+1)=N.
\end{eqnarray*}
In other words, $\Omega$ is an $N$-tuple of integers whose entries take values
of all the integers in $\cup_{i=1}^s[A_i,B_i]$.

We shall consider $\Omega(N)$
changing with $N$. Suppose that for each $N$, $\Omega(N)$ has corresponding $A_i(N)$, $B_i(N)$, for
a fixed $s$.
Assume also that $A_i(N),B_i(N),\Omega(N)_N-N$ have the following asymptotic
growth:
\begin{equation}
\label{cc2}
A_i(N)=\alpha_i N+o(N),\ \ B_i(N)=b_iN+o(N),
\end{equation}
where $0=\alpha_1<b_1<\ldots<\alpha_s<b_s$ are new parameters such that $\sum_{i=1}^{s}(b_i-\alpha_i)=1$. 
\item Suppose Part (1) of the assumption holds with $b_1\geq \gamma$ for $\gamma\in[0,1)$ as given in Assumption \ref{apew} (3).
\end{enumerate}

\end{assumption}

Recall that the Stieljes transform of $\tilde{\bm}$ is given by
\begin{eqnarray*}
\mathrm{St}_{\tilde{\bm}}(t)=\int_{\RR}\frac{\tilde{\bm}(ds)}{t-s}
\end{eqnarray*}

Following similar computations as in Section 4 of \cite{BL17}, we obtain that
\begin{proposition}\label{p38}Suppose Assumptions \ref{apew}, \ref{ap37} hold.
Let 
\begin{eqnarray}
F_{\kappa}(z,t):=\frac{t}{1-\kappa}-\frac{\kappa z}{(z-1)(1-\kappa)}+\frac{\kappa z}{n(1-\kappa)(1-\gamma)}\sum_{i\in[n]\cap I_2}\frac{y_ix}{1+y_ixz}\label{fk}
\end{eqnarray}
Then the following system of equations
\begin{eqnarray}
\begin{cases}F_{\kappa}(z,t)=\frac{\chi}{(1-\kappa)(1-\gamma)}-\frac{\gamma}{1-\gamma}\\\mathrm{St}_{\tilde{\bm}}(t)=\log(z) \end{cases}\label{ez}
\end{eqnarray}
have at most one pair of complex conjugate (non-real) roots in $z$. A point $(\chi,\kappa)$ is in the liquid region if and only if the above system of equations have one pair of complex conjugate roots. 
\end{proposition}

The idea to prove Proposition \ref{p38}, as in Section 4 of \cite{BL17}, is based on the fact that at each point in the frozen region, the density of the limit counting measure of partitions is either 0 or 1, while at each point in the liquid region, the density of the limit counting measure of partitions is in the open interval $(0,1)$. Using the identity 
(See e.g.~\cite[Lemma~4.2]{bk}):
  \begin{equation}
    \label{l36}
    f(x)=-\lim_{\epsilon\rightarrow 0+}\frac{1}{\pi}\Im(\mathrm{St}_{\mu}(x+i\epsilon)),
  \end{equation}
  where $\Im$ denote the imaginary part of a complex number, and $f(x)$ is the  continuous density of the measure $\mu$ on $\RR$ with respect to the Lebesgue measure; by computing the Stieljes transform of the limit counting measure $\tilde{\bm}$, we can express its density as the argument of the unique non-real solution in $z$ of (\ref{ez}) in the upper half plane, divided by $\pi$, in the case when (\ref{ez}) has exactly one pair of complex (non-real) conjugate roots. It follows that $(\chi,\kappa)$ is in the liquid region if and only if (\ref{ez}) have one pair of complex conjugate roots. 
  
  To prove that the system of equations (\ref{ez}) have at most one pair of complex conjugate (non-real) roots in $z$, as in Section 4 of \cite{BL17}, we express the solutions of $(\ref{ez})$ in $\RR$ as the $x$-coordinates of intersection points in $\RR^2$ of a straight line and a piecewise monotone curve which takes every value in $(-\infty,\infty)$ in each subinterval. Then we find at least $d-2$ real solutions for (\ref{ez}) (here $d$ is the total number of solutions of (\ref{ez}) in $z$). It follows that the system of equations (\ref{ez}) have at most one pair of complex conjugate (non-real) roots in $z$.

When the bottom boundary condition satisfies Assumption \ref{ap37}, let
\begin{eqnarray*}
\tilde{\alpha}_1&=&\frac{\alpha_1}{1-\gamma}=0\\
\tilde{\alpha}_i&=&\frac{\alpha_i-\gamma}{1-\gamma},\ \forall i\in[s]\setminus\{1\}\\
\tilde{b}_j&=&\frac{b_j-\gamma}{1-\gamma},\ \forall j\in[s].
\end{eqnarray*}

It is straightforward to check that 
\begin{eqnarray*}
\sum_{i=1}^{s}(\tilde{b}_s-\tilde{\alpha}_s)=1.
\end{eqnarray*}
Note that $\tilde{\bm}$,  the limit counting measure for the bottom boundary partitions, has density $1$ in each of the interval $[\tilde{\alpha}_i,\tilde{b}_i]_{i\in[s]}$, and $0$ everywhere else.
The Stieltjes transform can be computed explicitly from the definition:
  \begin{equation}
    \mathrm{St}_{\tilde{\bm}}(t)=\log\prod_{i=1}^{s}\frac{t-\tilde{\alpha}_i}{t-\tilde{b}_j}.\label{cst}
  \end{equation}

\begin{theorem}
  \label{fb}
  Let $c_i=\frac{1}{y_ix}$, for $i\in I_2\cap[n]$.
  The frozen boundary of the limit of a contracting bipartite graph $\mathcal{R}(\Omega,\check{a},X,Y,n)$
  satisfying Assumptions \ref{apew} and \ref{ap37} is a rational algebraic curve $C$ with an
  explicit parametrization $(\chi(t),\kappa(t))$ defined as follows:
  \begin{eqnarray*}
  \chi(t)&=&\left(t-\frac{J(t)}{J'(t)}\right)(1-\gamma)+\gamma\left(1-\frac{1}{J'(t)}\right),\\
    \kappa(t)&=&\frac{1}{J'(t)}
  \end{eqnarray*}
  where
  \begin{multline}
    J(t)=\Phi_s(t)\left[\frac{1}{\Phi_s(t)-1}-\frac{1}{n(1-\gamma)}\sum_{i\in
    I_2\cap[n]}\frac{1}{\Phi_s(t)+c_i}\right]
    \label{djt}
  \end{multline}
  and
  \begin{equation*}
    \Phi_s(t)=\frac{(t-\tilde{\alpha}_1)(t-\tilde{\alpha}_2)\cdots(t-\tilde{\alpha}_s)}{(t-\tilde{b}_1)(t-\tilde{b}_2)\cdots(t-\tilde{b}_s)}
  \end{equation*}
\end{theorem}

\begin{proof}

By (\ref{fk}), the first equation of (\ref{ez}) is linear in $t$. Solving $t$ from the first equation of (\ref{ez}), we obtain
\begin{eqnarray}
t=\frac{\kappa z}{z-1}-\frac{\kappa z}{n(1-\gamma)}\sum_{i\in[n]\cap I_2}\frac{1}{z+c_i}
+\frac{\chi-\gamma(1-\kappa)}{1-\gamma}\label{t1}
\end{eqnarray}
By (\ref{cst}), we may write the second equation of (\ref{ez}) as follows:
\begin{eqnarray}
z=\prod_{i=1}^s\frac{t-\tilde{\alpha}_i}{t-\tilde{b}_i}\label{t2}
\end{eqnarray}
%Plugging (\ref{t1}) into (\ref{t2}),  by Proposition \ref{p38}, the frozen boundary is given
%  by the condition that the following equation in the unknown $z$ has a double root:
 % \begin{equation}
 %   G\left(z,\frac{\chi}{(1-\kappa)(1-\gamma)}-\frac{%\gamma}{1-\gamma}\right)=z;\label{gzck}
%  \end{equation}
%  where 
%  \begin{equation}
 %   G\left(z,x\right)
  %  =\prod_{i=1}^s
 %   \frac{H(z;x,\tilde{a}_i)}{H(z;x,\tilde{b}_i)},
 % \end{equation}
%  and 
%  \begin{eqnarray*}
 % H(z,x,y)=\frac{\kappa z}{z-1}-\frac{\kappa %z}{n(1-\gamma)}\sum_{i\in[n]\cap I_2}\frac{1}{z+c_i}
%+(1-\kappa)x-\tilde{a}_i
 % \end{eqnarray*}

  Hence (\ref{ez}) is equivalent to:
  \begin{equation}
      \begin{cases}
	\Phi_s(t)&=z;\\
	(1-\kappa)F_\kappa(z,t) &=
	t-\left(\frac{\kappa z}{z-1}-\frac{\kappa z}{n(1-\gamma)}\sum_{i\in[n]\cap I_2}\frac{1}{z+c_i}\right)
=\frac{\chi-\gamma(1-\kappa)}{1-\gamma}
      \end{cases}\label{pkf}
    \end{equation}
    We plug the expression of $z$ from the first equation into the second
    equation, and note that the condition that the resulting equation has a
    double root is equivalent to the following system of equations
    \begin{equation*}
      \begin{cases}
	\frac{\chi-\gamma(1-\kappa)}{1-\gamma}=t-\kappa J(t),\\
	1=\kappa J'(t).
      \end{cases}
    \end{equation*}
    where $J(t)$ is defined by \eqref{djt}. Then the parametrization of the
    frozen boundary follows.
  \end{proof}
  
\begin{definition}\label{dfdu}Let $C\subset \RR P^2$ be a curve in the projective plane. The dual curve $C^{\vee}$ is defined by
\begin{eqnarray*}
C^{\vee}=\{(X,Y,Z)\in \RR P^2:  \{(x,y,z)\in \RR P^2: Xx+Yy+Zz=0\}\ \mathrm{is\ a\ tangent\ line\ of}\ C\}.
\end{eqnarray*}
The class of a curve is the degree of its dual curve.
\end{definition}

\begin{definition}(\cite{KO07})A degree $d$ real algebraic curve $C\subset \RR P^2$ is winding if:
\begin{enumerate}
\item it intersects every line $L\subset \RR P^2$ in at least $d-2$ points counting multiplicity; and
\item there exists a point $p_0\in \RR P^2$ called center, such that every line through $p_0$ intersects $C$ in $d$ points.
\end{enumerate}
The dual curve of a winding curve is called a cloud curve.
\end{definition}

\begin{proposition} \label{prop:cloud1}Let $s$ be the
  number of segments on the bottom boundary, and $m$
  be the number of distinct values of $c_i=\frac{1}{xy_i}$ in one period.
 Then
 \begin{enumerate}
 \item The frozen boundary $C$ is a cloud curve of class $(m+1)s$, if $\gamma<b_1$
 \item The frozen boundary $C$ is a cloud curve of class $(m+1)(s-1)$, if $\gamma=b_1$.
  \end{enumerate}

\end{proposition}

The result about the frozen boundary being a cloud curve extends the result of
\cite{BL17} for the contracting square-hexagon lattice.

\begin{proof}
By Definition \ref{dfdu},  we need to show that the dual curve $C^{\vee}$ has degree $(m+1)s$ or $(m+1)(s-1)$ and is
  winding.

  We apply the classical formula to obtain from a parametrization $(x(t),
  y(t))$ of the curve $C$ defining the frozen boundary, another
  one for its dual $C^\vee$, 
    $(x^\vee(t),y^\vee(t))$:
    \begin{equation*}
      x^{\vee}=\frac{y'}{yx'-xy'},\quad
      y^{\vee}=-\frac{x'}{yx'-xy'}.
    \end{equation*}
    and obtain that 
    the dual curve $C^{\vee}$ is
  given in the following parametric form:
  \begin{equation}
    C^{\vee}=\left\{\left(-\frac{1}{(1-\gamma)t+\gamma},-\frac{(1-\gamma)J(t)+\gamma}{(1-\gamma)t+\gamma}\right)\ ;\
    t\in\mathbb{C}\cup\{\infty\}\right\}.
    \label{dual}
  \end{equation}
  from which we can see that its degree is $(m+1)s$ (resp.\ $(m+1)(s-1)$) if $\gamma<b_1$ (resp.\ $\gamma=b_1$). To show that $C^{\vee}$ is
  winding, we need to look at real intersections with straight lines.
  
  Let
  \begin{eqnarray*}
  \tilde{t}=(1-\gamma)t+\gamma;\qquad
  \tilde{J}(\tilde{t})=(1-\gamma)J(t)+\gamma.
  \end{eqnarray*}
  Then we can reparametrize the curve $C^\vee$ by 
  \begin{equation}
    C^{\vee}=\left\{\left(-\frac{1}{\tilde{t}},-\frac{\tilde{J}(\tilde{t})}{\tilde{t}}\right)\ ;\
    t\in\mathbb{C}\cup\{\infty\}\right\}.
    \label{dur}
  \end{equation}

  First, from Equation~\eqref{dur}, one sees that the first coordinate $x^\vee$ of
  the dual curve $C^\vee$ and the parameter $\tilde{t}$ are linked by the simple
  relation $x^\vee \tilde{t}=-1$.
  
  Using this relation to eliminate $t$ from the expression of the second
  coordinate, we obtain that the points $(x^\vee,\tilde{t})$ on the dual curve satisfy the
  following implicit equation:
  \begin{equation*}
    y^{\vee}=x^{\vee} \tilde{J}\left(-\frac{1}{x^\vee}\right).
  \end{equation*}

  The points of intersection $(x^\vee(\tilde{t}),y^\vee(\tilde{t}))$ of the dual curve with a straight line
  of the form $y^\vee=cx^\vee+d$ have a parameter $t$ satisfying:
  \begin{equation}
    (c-d\tilde{t})=\tilde{J}(\tilde{t})
    \label{eq:intersect_dual}
  \end{equation}
  Note that
  \begin{eqnarray*}
  \Phi_s(t)=\frac{(\tilde{t}-\gamma)(\tilde{t}-\alpha_2)\cdot\ldots\cdot(\tilde{t}-\alpha_s)}{(\tilde{t}-b_1)(\tilde{t}-b_2)\cdot\ldots\cdot(\tilde{t}-b_s)}:=\tilde{\Phi}_s(\tilde{t})
  \end{eqnarray*}
 and
 \begin{eqnarray*}
 \tilde{J}(\tilde{t})=\frac{|I_1\cap[n]|}{n}+\frac{1-\gamma}{\tilde{\Phi}_s(\tilde{t})-1}+\frac{1}{n}\sum_{i\in I_2\cap [n]}\frac{c_i}{\tilde{\Phi}_s(\tilde{t})+c_i}
 \end{eqnarray*}

Then the conclusion that the frozen boundary $C$ is a cloud curve follows from the same arguments as in the proof of Proposition 5.4 in \cite{BL17}.
\end{proof}

\begin{example}\label{ex3}Suppose that we have a sequence of contracting bipartite graphs satisfying:
\begin{itemize}
\item the boundary condition as in Assumption \ref{ap37} satisfies $s=2$, $\alpha_1=0$, $b_1=\frac{2}{3}$, $\alpha_2=1$, $b_2=\frac{4}{3}$; 
\item the edge weights as in Assumption \ref{apew} satisfy
\begin{itemize}
\item $n=3$, $\gamma=\frac{1}{3}$; and
\item  $x_1=x_2=1$, $x_3=0$; and
\item $y_1=0$, $y_2=1$, $y_3=2$; and
\item $a_1=1$, $a_2=a_3=0$.
\end{itemize}
\end{itemize}
\begin{figure}
\includegraphics[trim=3cm 8cm 3.3cm 8cm,clip,width=.6\textwidth]{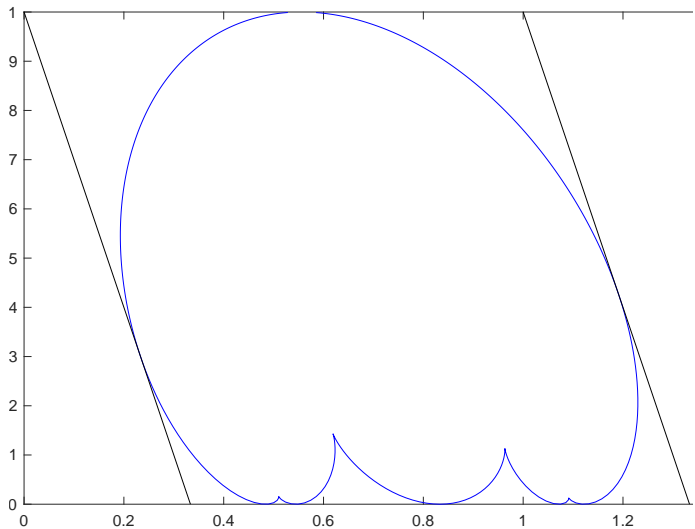}
\caption{Frozen boundary of Example \ref{ex3}}\label{fig:frozen1}
\end{figure}
Then by Theorem \ref{fb},
\begin{eqnarray*}
\Phi_2(t)&=&\frac{t(t-1)}{\left(t-\frac{1}{2}\right)\left(t-\frac{3}{2}\right)}.\\
J(t)&=&\Phi_2(t)\left[\frac{1}{\Phi_2(t)-1}-\frac{1}{2}\left(\frac{1}{\Phi_2(t)+1}+\frac{1}{\Phi_2(t)+\frac{1}{2}}\right)\right].
\end{eqnarray*}
The parametric equation for the frozen boundary is then given by
\begin{eqnarray*}
  \chi(t)&=&\frac{2}{3}\left(t-\frac{J(t)}{J'(t)}\right)+\frac{1}{3}\left(1-\frac{1}{J'(t)}\right),\\
    \kappa(t)&=&\frac{1}{J'(t)}
  \end{eqnarray*}
See Figure \ref{fig:frozen1} for the frozen boundary of the limit shape of perfect matchings. The region covered by the contracting bipartite graphs is the trapezoid region $(\chi,\kappa)$ bounded by $\kappa=0$, $\kappa=1$,$\chi=0$, $\kappa=-3\chi+4$. From the analysis above we see that the triangular region $(\chi,\kappa)$ bounded by $\kappa=0$, $\chi=0$, $\kappa=-3\chi+1$ must be frozen.
\end{example}

\section{Disconnected Liquid Regions}\label{sect:s4}

In this section, we prove the existence of multiple disconnected liquid regions for the limit shape of perfect matchings on any contracting square-hexagon lattices with certain edge weights, extending the results in \cite{ZL18}. The main theorem proved in this section is Theorem \ref{thm:m2}. We shall start with a few assumptions.

\begin{assumption}\label{ap423}Suppose the edge weights $(x_1,\ldots,x_N)$ satisfies Assumption \ref{apew} (1)(2). Moreover, 
\begin{itemize}
\item $x_1>x_2>\ldots>x_n>0$; and
\item $N$ is an integral multiple of $n$.
\end{itemize}
\end{assumption}

Let $\sigma_0$ be an arbitrary permutation in $\Sigma_N$ such that
\begin{eqnarray}
x_{\si_0(1)}\geq x_{\si_0(2)}\geq\ldots\geq x_{\si_0(N)}.\label{sz}
\end{eqnarray}

\begin{assumption}\label{ap428}Assume $x_1,\ldots,x_N$ satisfy Assumption \ref{apew} (1)(2) and Assumption \ref{ap423}.

Let $s\in[N]$.
Assume there exists positive integers $K_1,K_2,\ldots K_s$, such that 
\begin{enumerate}
\item $\sum_{t=1}^s K_t=N$;
\item
\begin{eqnarray}
\mu_1>\ldots>\mu_s\label{mi}
\end{eqnarray}
are all the distinct elements in $\{\lambda_1,\lambda_2,\ldots,\lambda_N\}$.
\item
\begin{eqnarray*}
&&\lambda_1=\lambda_2=\ldots=\lambda_{K_s}=\mu_1;\\
&&\lambda_{K_s+1}=\lambda_{K_s+2}=\ldots=\lambda_{K_s+K_{s-1}}=\mu_2;\\
&&\ldots\\
&&\lambda_{\sum_{t=2}^{s}K_t}=\lambda_{1+\sum_{t=2}^{s}K_t}=\ldots=\lambda_{\sum_{t=1}^{s}K_t}=\mu_s;
\end{eqnarray*}
 \item Let 
\begin{eqnarray}
J_i=\{t\in[s]:\exists p\in[N],\ \mathrm{s.t.}\ x_{\sigma_0(p)}=x_i,\mathrm{and}\ \lambda_p=\mu_t\}\label{ji}
\end{eqnarray}
\begin{enumerate}
\item If $1\leq i<j\leq n$, $\ell\in J_i$, and $t\in J_j$, then $\ell<t$.
\item For any $p,q$ satisfying $1\leq p\leq s$ and $1\leq q\leq s$, and $q>p$
\begin{eqnarray*}
L_1N \leq \mu_p-\mu_q\leq L_2N
\end{eqnarray*}
where $L_1$, $L_2$ are sufficiently large constants independent of $N$.
\item $s$ and $n$ are fixed as $N\rightarrow\infty$.
\end{enumerate}
\end{enumerate}
\end{assumption}

\noindent{\textbf{Remark.}} When the edge weights are periodic and satisfy Assumption 2.1 (2), the condition that $\exists p\in[N],\ \mathrm{s.t.}\ x_{\sigma_0(p)}=x_i$ is equivalent to the following condition:
\begin{eqnarray*}
\exists\ p\in \left[\frac{(i-1)N}{n}+1,\frac{i N}{n}\right]\cap \NN.
\end{eqnarray*}
Hence $J_i$ can also be defined as follows
\begin{eqnarray*}
J_i=\left\{t\in[s]:\exists\ p\in \left[\frac{(i-1)N}{n}+1,\frac{i N}{n}\right]\cap \NN,\ \mathrm{s.t.}\ \lambda_p=\mu_t\right\}
\end{eqnarray*}

An  equivalent condition of Assumption \ref{ap428} 4(a)  is as follows: 
\begin{eqnarray*}
\left\{\frac{iN}{n}: i\in [n]\right\}\subseteq \left\{\sum_{t=j}^s K_t: j \in[s]\right\}.
\end{eqnarray*}

\begin{assumption}\label{ap32}Assume $x_{1,N}=x_1>0$ and $(x_{2,N},\ldots,x_{n,N})$ changes with $N$. Assume that for each fixed $N$,  $(x_{1,N},\ldots,x_{n,N})$ satisfy Assumption \ref{ap423}. Suppose that Assumption \ref{ap428} holds. Moreover, assume that
 \begin{eqnarray*}
\liminf_{N\rightarrow\infty} \frac{\log\left(\min_{1\leq i<j\leq n}\frac{x_{i,N}}{x_{j,N}}\right)}{N}\geq \alpha>0,
\end{eqnarray*}
where $\alpha$ is a sufficiently large positive constant independent of $N$.
\end{assumption}

\begin{theorem}\label{thm:m2}
Suppose  Assumptions \ref{ap37}(1), \ref{ap423}, \ref{ap428} and \ref{ap32}  hold. Then the frozen boundary consists of $n$ disjoint cloud curves.
\end{theorem}

The conclusion of Theorem \ref{thm:m2} under the further condition that $|I_2\cap[n]|\in\{0,1\}$ was proved in Theorem 2.20 of \cite{ZL18}. Here we prove the theorem for all the contracting square-hexagon lattice when edge weights satisfy the Assumptions \ref{ap423}, \ref{ap32} and bottom boundary conditions satisfy Assumptions \ref{ap37}(1) and \ref{ap428}.

For $i\in[n]$, let $J_i$ be defined as in (\ref{ji}).
Under Assumptions \ref{ap423} and \ref{ap428}, assume that 
\begin{eqnarray*}
J_i=\left\{\begin{array}{cc}\{d_i,d_i+1,\ldots,d_{i+1}-1\}& \mathrm{if}\ 1\leq i\leq n-1\\ \{d_n, d_n+1,\ldots,s\}& \mathrm{if}\ i=n \end{array}\right.
\end{eqnarray*}
where $d_1,\ldots,d_n$ are positive integers satisfying
\begin{eqnarray*}
1=d_1<d_2<\ldots <d_n\leq s
\end{eqnarray*}
Let
\begin{eqnarray*}
d_{n+1}:=s+1.
\end{eqnarray*}

If $i\in[n]$, for $0\leq k\leq d_{i+1}-d_i-1:=D_i$, let
\begin{eqnarray}
\beta_{i,k}&=&n\left(\alpha_1+\sum_{l=2}^{s-d_i-k+1}(\alpha_l-b_{l-1})\right)+n-i+1-n\left(\sum_{l=s-d_i-k+1}^{s-d_i+1}(b_l-\alpha_l)\right)\label{dbt}\\
\gamma_{i,k}&=&n\left(\alpha_1+\sum_{l=2}^{s-d_i-k+1}(\alpha_l-b_{l-1})\right)+n-i+1-n\left(\sum_{l=s-d_i-k+2}^{s-d_i+1}(b_l-\alpha_l)\right).\label{drm}
\end{eqnarray}
For $i\in[n]$, let
 \begin{equation}
    \Psi_i(t_i)=\frac{(t_i-\beta_{i,0})(t_i-\beta_{i,1})\cdots(t_i-\beta_{i,D_i})}{(t_i-\gamma_{i,0})(t_i-\gamma_{i,1})\cdots(t_i-\gamma_{i,D_i})};\label{psi}
\end{equation}
and for $j\in[s]$, let $a_j,b_j$ be given by (\ref{cc2}).

For
$\kappa\in(0,1)$ , $2\leq i\leq n$, let:
\begin{equation*}
F_{i,\kappa}(z,t_i):=
\frac{z}{(1-\kappa)n}\left(\frac{t_i}{z}-\frac{1}{z-1}-\frac{n-i}{z}\right)+\frac{z}{n(z-1)}+\frac{n-i}{n}.
\end{equation*}
and 
\begin{equation*}
F_{1,\kappa}(z,t_1):=
\frac{z}{(1-\kappa)n}\left(\frac{t_1}{z}-\frac{1}{z-1}-\frac{n-1}{z}+\kappa\sum_{r\in I_2\cap [n]}\frac{y_{r}x_1}{1+y_{r}x_1z}\right)+\frac{z}{n(z-1)}+\frac{n-1}{n}.
\end{equation*}

Let $C_1$ be the curve consisting of all the points $(\chi,\kappa)\in [0,\infty)\times [0,1]$ such that the following system of equations has a double root in $z$
\begin{equation}
      \begin{cases}
	\Psi_1(t_1)=z;\\
	n(1-\kappa)F_{1,\kappa}(z) =
	t_1-\kappa \left[ \frac{1}{z-1}+(n-l)+\sum_{j=1}^{m}\frac{n_j\gamma_j}{z+\gamma_j}\right]
=n\chi.
      \end{cases}\label{s1}
      \end{equation}
where $\gamma_1,\ldots,\gamma_m$ are all the distinct values in $\left\{\frac{1}{y_ix}\right\}_{i\in I_2\cap[n]}$; for $j\in[m]$, $n_j$ is the number of $i\in I_2\cap [n]$ such that $\frac{1}{y_ix}=\gamma_j$; and 
\begin{eqnarray*}
l=I_2\cap[n].
\end{eqnarray*}

For $i\in\{2,3,\ldots,n\}$, let $C_i$ be the curve consisting of all the points $(\chi,\kappa)\in [0,\infty)\times [0,1]$ such that the following system of equations has a double root in $z$ 
\begin{equation}
      \begin{cases}
	\Psi_i(t_i)=z;\\
	n(1-\kappa)F_{i,\kappa}(z) =
	t_i-\kappa
\left[(n-i+1)+\frac{1}{z-1}\right]
 =n\chi.
      \end{cases}\label{s2}
    \end{equation}

Then we have the following lemma
\begin{lemma}\label{l45}Let $i\in[n]$. The curve $C_i$ has an explicit parametrization given by 
  \begin{equation*}
    \chi_i(t_i)=\frac{1}{n}\left[t_i-\frac{J_i(t_i)}{J_i'(t_i)}\right],\quad
    \kappa_i(t_i)=\frac{1}{J_i'(t_i)},
  \end{equation*}
  where for $2\leq i\leq n$;
\begin{eqnarray*}  
J_i(t_i)=(n-i+1)+\frac{1}{\Psi_i(t_i)-1};
\end{eqnarray*}
for $i=1$
\begin{eqnarray*}
J_1(t_1)=\frac{1}{\Psi_1(t_1)-1}+n-l+\sum_{i\in [n]\cap I_2}\frac{c_j}{\Psi_1(t_1)+c_j}
\end{eqnarray*} 
where $c_j=\frac{1}{y_jx_1}$ for $j\in [n]\cap I_2$.
Moreover, the curves $C_1,\ldots,C_n$ are disjoint cloud curves. 
\end{lemma}

\begin{proof}For each $i\in[n]$, the parametric equation of $C_i$ and the fact that $C_i$ is a cloud curve of class $D_i$ follow from the same arguments as in the proof of Theorem 7.10 in \cite{ZL18}.

 We need to show that $C_1,\ldots, C_n$ are disjoint. Note that $C_i$ is characterized by the condition that the system (\ref{s1}) or (\ref{s2}) of equations have double roots in $z$. When $N$ is an integer multiple of $n$, recall that the partition $\phi^{(i,\si)}(N)\in \YY_{\frac{N}{n}}$ was defined as in (\ref{pis}), and $\si_0\in \Si_N$ was defined as in (\ref{sz}). When the bottom boundary conditions satisfy Assumptions  
\ref{ap37}(1) and \ref{ap428}, for each $i\in[n]$, the counting measures for $\phi^{(i,\si_0)}(N)$ converge weakly to a limit measure, denoted by $\bm_i$, as $N\rightarrow\infty$. The measure $\bm_i$ has density 1 in $\cup_{k=0}^{D_i}[\beta_{i,k},\gamma_{i,k}]$, and density 0 elsewhere; see Lemma 4.8 of \cite{ZL18}.
Note that the supports for $\bm_i$, $i\in[n]$, are pairwise disjoint.

 We make a change of variables in (\ref{s1}) as follows.
  \begin{eqnarray*}
 \begin{cases}
 \tilde{\chi}_1=\chi+\frac{n-1}{n}\\
 \tilde{\kappa}_1=\kappa
 \end{cases},
 \end{eqnarray*}
Then after the change of variables, (\ref{s1}) becomes the same as (\ref{pkf}) when $\gamma=\frac{n-1}{n}$, $D_1=s-1$ and $(\beta_{1,D_1},\gamma_{1,D_1},\ldots,\beta_{1,0},\gamma_{1,0})=(\tilde{\alpha}_1,\tilde{b}_1,\ldots,\tilde{\alpha}_s,\tilde{b}_s)$, and $x_1=x$. Let $\tilde{C}_1=C_1+\left(\frac{n-1}{n},0\right)$. Then $\tilde{C}_1$ is the frozen boundary of the limit shape of perfect matchings a contracting bipartite graph, hence it is restricted in the region bounded by
 \begin{eqnarray*}
 &&\tilde{\kappa}_1=0;\qquad\tilde{\kappa}_1=1;\\
 &&\tilde{\chi}_1=\beta_{1,D_1}(1-\gamma)=\frac{\beta_{1,D_1}}{n}\\
 &&\tilde{\chi}_1=\gamma_{1,0}(1-\gamma)+\gamma=\frac{\gamma_{1,0}}{n}+\frac{n-1}{n}
 \end{eqnarray*}
 Hence $C_1$ is restricted in the region bounded by
 \begin{eqnarray*}
 &&\kappa=0;\qquad \kappa=1\\
 &&\chi=\frac{\beta_{1,D_1}}{n}-\frac{n-1}{n}\\
 &&\chi=\frac{\gamma_{1,0}}{n}
 \end{eqnarray*}
 
 We also make a change of variables in (\ref{s2}). For $i\in\{2,3,\ldots,n\}$, let
 \begin{eqnarray*}
 \begin{cases}
 \tilde{\chi}_i=n\chi+\kappa(n-i)\\
 \tilde{\kappa}_i=\kappa
 \end{cases},
 \end{eqnarray*}
 and let $\tilde{C}_i$ be the corresponding curve in the new coordinate system $\tilde{\chi}_i,\tilde{\kappa}_i$. Then $\tilde{C}_i$ is the frozen boundary of a uniform dimer model on contracting hexagon lattice with boundary condition given by $\bm_i$. Then the fact that $\tilde{C}_i$ and $C_i$ are cloud curves follows from Proposition 5.4 of \cite{BL17}. Hence $\tilde{C}_i$ is restricted in the region bounded by 
 \begin{eqnarray*}
 &&\tilde{\kappa}_1=0;\qquad
 \tilde{\kappa}_1=1\\
 &&\tilde{\chi}_1=\beta_{i,D_i}; \qquad
\tilde{\chi}_1=\gamma_{i,0}
 \end{eqnarray*}
 Hence $C_i$ is restricted in the region bounded by 
 \begin{eqnarray*}
 &&\kappa=0;\qquad
 \kappa=1\\
 &&\chi=\frac{\beta_{i,D_i}-\kappa(n-i)}{n};\\ 
&&\chi=\frac{\gamma_{i,0}-\kappa(n-i)}{n}
 \end{eqnarray*}
 It is straightforward to check that when Assumption \ref{ap428}(4)(b) holds with constant $L_1$ sufficiently large, the regions described above containing different $C_i$ are disjoint; therefore $C_i$ are disjoint. 
\end{proof}

\noindent\textbf{Proof of Theorem \ref{thm:m2}.}
We consider a contracting square-hexagon lattice with edge weights satisfying Assumption \ref{apew}, $\gamma=\frac{n-1}{n}$, and boundary partition given by
\begin{eqnarray*}
\omega:=\left(\phi^{(1,\si_0)}(N),0^{\frac{N(n-1)}{n}}\right)\in\YY_{N}. 
\end{eqnarray*}
i.e., $\omega$ is a partition with $N$ components, the $\frac{N}{n}$ components at the beginning are those of $\phi^{(1,\si_0)}(N)$'s, and the remaining $\frac{N(n-1)}{n}$ components are 0. Let $\bm_{\omega}$ be the limit counting measure for $\omega$, as $N\rightarrow\infty$, and let $\tilde{\bm}_{\omega}$ be the limit counting measure for $\phi^{(1,\si_0)}(N)$.
Let $\kappa\in(0,1)$, and $\bm_{\omega}^{\kappa}$ be the limit counting measure for the partitions on the $\left\lfloor\frac{2\kappa N}{n}\right\rfloor$th row, counting from the bottom. Let $\tilde{\bm}_{\omega}^{\kappa}$ be the limit counting measure of the partitions obtained from the partitions on the $\left\lfloor\frac{2\kappa N}{n} \right\rfloor$th row by removing the $\frac{(1-\kappa)(n-1)N}{n}$ 0's in the end. By Proposition \ref{p36}, we obtain
\begin{eqnarray*}
\sum_{j=0}^{\infty}\frac{\int_{\RR}x^j \tilde{\bm}_{\omega}^{\kappa}(dx)}{s^{j+1}}=-\frac{1}{2\pi \mathbf{i}}\oint_1\frac{dz}{z}\log\left(1-\frac{F_{\kappa}(z)}{s}\right).
\end{eqnarray*}
Following similar computations as in Section 7 of \cite{ZL18}, we obtain that 
\begin{eqnarray}
\mathrm{St}_{\tilde{\bm}^{\kappa}_{\omega}}\left(x\right)=\log (z^{\kappa}(x))\label{sl}
\end{eqnarray}
where $z^{\kappa}(x)$ is the solution of $F_{\kappa}(z)=x$ in a neighborhood of $1$ when $x$ is in a neighborhood of infinity. When $x$ is not in a neighborhood of infinity but outside the support of $\tilde{\bm}_{\omega}^{\kappa}$, the identity (\ref{sl}) is obtained by analytic extension. 

We claim that if complex roots exists for (\ref{pkf}) with $x=\frac{\chi-\gamma(1-\kappa)}{(1-\gamma)(1-\kappa)}$ and $\gamma=\frac{n-1}{n}$, then $z^{\kappa}(x)$ cannot be real. 

To see why that is true, from (\ref{sl}) we obtain
\begin{eqnarray*}
z^{\kappa}(x)=\mathrm{exp}\left(\int_{\RR}\frac{\tilde{\bm}_{\omega}^{\kappa}[ds]}{x-s}\right);
\end{eqnarray*}
and 
\begin{eqnarray*}
z^{\kappa}(x)=\exp\left(\int_{\RR}\frac{(x-s-\mathbf{i}\epsilon) \tilde{\bm}_{\omega}^{\kappa}[ds]}{\left(x-s\right)^2+\epsilon^2}\right)
\end{eqnarray*}
Therefore $\Im[z^{\kappa}(x+\mathbf{i}\epsilon)]<0$ when $\epsilon$ is a small positive number. However, when complex roots exist for (\ref{pkf}), for real root $s(x)$, Lemma 7.9 of \cite{ZL18} implies that $\Im[s(x+\mathbf{i}\epsilon)]\geq 0$ when $\epsilon$ is a small positive number. This implies that when complex roots exist for (\ref{sl}), $z^{\kappa}(x+\mathbf{i}\epsilon)$ cannot be real.

From expression (7.2) of \cite{ZL18}, we obtain
\begin{eqnarray*}
\mathrm{St}_{\bm^{\kappa}}(x)=\sum_{i=1}^{n}\log(z_i^{\kappa}(x))
\end{eqnarray*}
where $\bm^{\kappa}$ is the limit counting measure for partitions corresponding to dimer configurations at level $\kappa$ on a contracting square-hexagon lattice with edge weights and bottom boundary conditions satisfying Assumptions \ref{ap423}, \ref{ap428}, \ref{ap32} and \ref{ap37}(1). Note that
\begin{itemize}
\item When $x$ is in a neighborhood of $\infty$, $z_1(x)$ is the root of (\ref{s1}) in a neighborhood of (\ref{s1}) when $x=n\chi$. Therefore
\begin{eqnarray*}
z_1^{\kappa}(x)=z^{\kappa}\left(nx+\frac{\kappa(n-1)}{1-\kappa}\right); 
\end{eqnarray*}
and
\item when $x$ is in a neighborhood of $\infty$ and $2\leq i\leq n$, $z_i(x)$ is the root of (\ref{s2}) in a neighborhood of 1. It follows from Section 7.1 of \cite{ZL18}, that when (\ref{s2}) has a pair of complex conjugate roots (non-real) in $z$, $z_i(x)$ cannot be real; and
\item By Lemma \ref{l45}, the regions when each one of (\ref{s1}) or (\ref{s2}) has a pair of complex conjugate roots are disjoint. That is, at each point of the trapezoid domain occupied by the square-hexagon lattice, at most one of the $z_i$'s ($1\leq i\leq n$) is not real.  
\end{itemize}

Therefore, under the assumptions of the theorem, the frozen boundary is given by the condition that one of the  $n$ equation systems in (\ref{s1}), (\ref{s2}) has double roots in $z$. Then the theorem follows from Lemma \ref{l45}.
$\hfill\Box$

\begin{example}\label{ex4}Consider a contracting square-hexagon lattice with period $1\times 2$. Let $x_1=1$, and $\frac{x_2}{x_1}\leq e^{-\alpha N}$; $y_1=4$ and $y_2=1/4$. Assume $N$ is an integer multiple of $6$. Assume the boundary partition $\lambda(N)$ satisfies
\begin{eqnarray*}
&&\lambda_1(N)=\lambda_2(N)=\ldots=\lambda_{\frac{N}{4}}(N)=\mu_1(N)\\
&&\lambda_{\frac{N}{4}+1}(N)=\lambda_{\frac{N}{4}+2}(N)=\ldots=\lambda_{\frac{N}{2}}(N)=\mu_2(N)\\
&&\lambda_{\frac{N}{2}+1}(N)=\lambda_{\frac{N}{2}+2}(N)=\ldots=\lambda_{\frac{2N}{3}}(N)=\mu_3(N)\\
&&\lambda_{\frac{2N}{3}+1}(N)=\lambda_{\frac{N}{2}+2}(N)=\ldots=\lambda_{\frac{5N}{6}}(N)=\mu_4(N)\\
&&\lambda_{\frac{5N}{6}+1}(N)=\lambda_{\frac{N}{2}+2}(N)=\ldots=\lambda_{N}(N)=\mu_5(N)=0.
\end{eqnarray*}
For $1\leq j\leq 5$, let
\begin{eqnarray*}
r_j=\lim_{N\rightarrow\infty}\frac{\mu_j(N)}{N}.
\end{eqnarray*}
Note that $r_5=0$. Then
\begin{eqnarray*}
\Psi_1(t_1)&=&\frac{\left(t_1-2r_1-\frac{3}{2}\right)\left(t_1-2r_2-1\right)}{\left(t_1-2r_1-2\right)\left(t_1-2r_2-\frac{3}{2}\right)}; \\
\Psi_2(t_2)&=&\frac{t_2\left(t_2-2r_4-\frac{1}{3}\right)\left(t_2-2r_3-\frac{2}{3}\right)}{(t_2-\frac{1}{3})\left(t_2-2r_4-\frac{2}{3}\right)(t_2-2r_3-1)};\\
J_1(t_1)&=&\frac{1}{\Psi_1(t_1)-1}+\frac{1}{\Psi_1(t_1)+1}+\frac{1}{2\Psi_1(t_1)+1}; \\
J_2(t_2)&=&\frac{1}{\Psi_2(t_2)-1}+1.
\end{eqnarray*}
Then the frozen boundary is a union of the following two parametric curves.
\begin{eqnarray*}
\begin{cases}
\chi_1(t_1)=\frac{1}{2}\left[t_1-\frac{J_1(t_1)}{J_1'(t_1)}\right]\\
\kappa_1(t_1)=\frac{1}{J_1'(t_1)}
\end{cases}
\end{eqnarray*}
and
\begin{eqnarray*}
\begin{cases}
\chi_2(t_2)=\frac{1}{2}\left[t_2-\frac{J_2(t_2)}{J_2'(t_2)}\right]\\
\kappa_2(t_2)=\frac{1}{J_2'(t_2)}
\end{cases}
\end{eqnarray*}
For $(r_1,r_2,r_3,r_4)=(6,5,2,1)$, see Figure \ref{fig:fbr} for a picture of the frozen boundary.

\begin{figure}
 \includegraphics[width=0.9\textwidth]{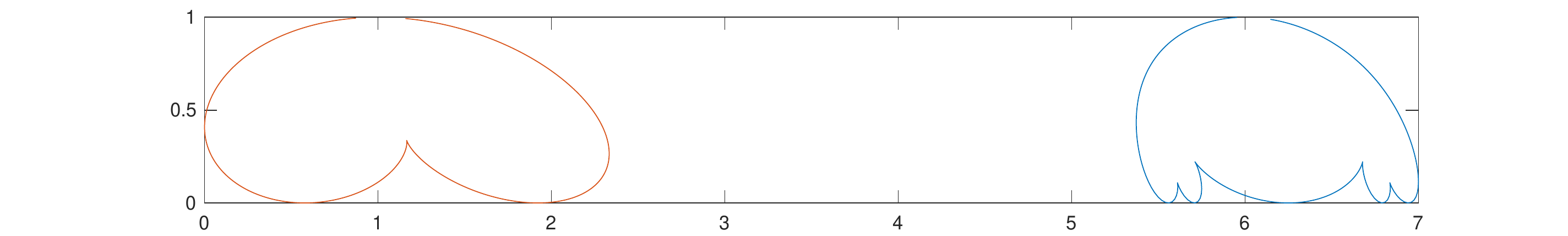}
\caption{Frozen boundary of Example \ref{ex4}}
\label{fig:fbr}
\end{figure}
\end{example}

\bigskip
\noindent\textbf{Acknowledgements.} ZL acknowledges support from National Science Foundation DMS 1608896 and Simons Foundation 638143. The author thanks the anonymous reviewer for helpful advice to improve the readability of the paper.

\bibliography{fpmm}
\bibliographystyle{plain}

\end{document}